\documentclass[review, nopreprintline]{elsarticle}
\usepackage[T1]{fontenc}
\usepackage{tikz}
\usetikzlibrary{arrows,arrows.meta,positioning,calc}
\usepackage{amsthm,amsmath,amssymb,amsfonts,amscd,latexsym}
\usepackage[top=3cm, bottom=3cm, left=3cm, right=2cm]{geometry}

\biboptions{sort&compress}

\newtheorem{theorem}{Theorem}[section]
\newtheorem{conjecture}[theorem]{Conjecture}
\newtheorem{corollary}[theorem]{Corollary}
\newtheorem{proposition}[theorem]{Proposition}
\newtheorem{problem}[theorem]{Problem}

\newtheorem{lemma}[theorem]{Lemma}

\theoremstyle{definition}
\newtheorem{definition}[theorem]{Definition}

\theoremstyle{remark}
\newtheorem{remark}[theorem]{Remark}
\newtheorem{example}[theorem]{Example}

\newcommand{\dist}{\operatorname{dist}}
\newcommand{\diam}{\operatorname{diam}}

\newcommand{\NN}{\mathbb{N}_0}
\newcommand{\CC}{\mathbb{C}}

\let\ol\overline

\newcommand{\card}{\operatorname{card}}

\numberwithin{equation}{section}

\makeatletter

\renewcommand{\p@enumii}{}

\makeatother

\date{}
\journal{}

\begin{document}

\begin{frontmatter}
\title{Uniqueness of best proximity pairs and rigidity of semimetric spaces}

\author[1]{Oleksiy Dovgoshey\corref{cor1}}
\ead{oleksiy.dovgoshey@gmail.com}

\author[2]{Ruslan Shanin}
\ead{ruslanshanin@gmail.com}

\address[1]{Department of Theory of Functions, Institute of Applied Mathematics and Mechanics of NASU, Dobrovolskogo str. 1, Slovyansk 84100, Ukraine and Institut f\"{u}r Mathematik Universit\"{a}t zu L\"{u}beck, Ratzeburger Allee 160, D-23562 L\"{u}beck, Deutschland}

\address[2]{Department of Mathematical Analysis, Odesa I.~I.~Mechnikov National University, Dvoryanskaya str., 2, Odesa 65082, Ukraine}

\cortext[cor1]{Corresponding author}

\begin{abstract}
For arbitrary semimetric space \((X, d)\) and disjoint proximinal subsets \(A\), \(B\) of \(X\) we define the proximinal graph as a bipartite graph with parts \(A\) and \(B\) whose edges \(\{a, b\}\) satisfy the equality \(d(a, b) = \dist(A, B)\). We characterize the semimetric spaces whose proximinal graphs have at most one edge and the semimetric spaces whose proximinal graphs have the vertices with degree at most \(1\) only. This allows us to describe the necessary and sufficient conditions for uniqueness of the best proximity pairs and best approximations.
\end{abstract}

\end{frontmatter}

\section{Introduction}

Let \(X\) be a set. A \emph{semimetric} on \(X\) is a function \(d\colon X \times X \to [0, \infty)\) such that \(d(x, y) = d(y, x)\) and \((d(x, y) = 0) \Leftrightarrow (x = y)\) for all \(x\), \(y \in X\). A pair \((X, d)\), where \(d\) is a semimetric on \(X\), is called a \emph{semimetric space} (see, for example, \cite[p.~7]{Blumenthal1953}). A semimetric \(d\) is a \emph{metric} if the \emph{triangle inequality}
\[
d(x, y) \leqslant d(x, z) + d(z, y)
\]
holds for all \(x\), \(y\), \(z \in X\). In this paper, we only consider the nonempty semimetric and metric spaces.

The following definition is well-known for the case of metric spaces. See, for example, Definition~2.1 in~\cite{Sin1974}.

\begin{definition}\label{d1.1}
Let \((X, d)\) be a semimetric space. A set \(A \subseteq X\) is said to be \emph{proximinal} in \((X, d)\) if, for every \(x\in X\), there exists \(a_0 = a_0(x) \in A\) such that
\begin{equation*}
d(x,a_0) = \inf\{d(x, a)\colon a\in A\}.
\end{equation*}
The point \(a_0\) is called a \emph{best approximation} to \(x\) in \(A\).
\end{definition}

\begin{remark}
In~\cite{Sin1974} Ivan Singer wrote: ``The term <<proximinal>> set (a combination of <<proximity>> and <<minimal>>) was proposed by R. Killgrove and used first by R. R. Phelps \cite{Phe1957PotAMS}.''
\end{remark}

For nonempty subsets \(A\) and \(B\) of a semimetric space \((X,d)\), we define a distance from \(A\) to \(B\) as
\begin{equation}\label{e1.2}
\dist(A,B) := \inf\{d(a,b)\colon a\in A\ \text{and}\ b\in B\}.
\end{equation}
If \(A\) is a one-point set, \(A = \{a\}\), then we write \(\dist(a, B)\) instead of \(\dist(\{a\}, B)\).

The next is a semimetric modification of Definition~1.1 from~\cite{ref12}.

\begin{definition}\label{d1.2}
Let \((X, d)\) be a semimetric space, and let \(A\), \(B \subseteq X\) be nonempty. A pair \((a_0, b_0) \in A \times B\) is called a \emph{best proximity pair} for the sets \(A\) and \(B\) if \(d(a_0, b_0) = \dist(A, B)\).
\end{definition}

Some results connected with existence of the best approximations and the best proximity pairs in metric spaces can be found in \cite{ref9,ref10,ref11, ref12, Sch1985, Phe1957PotAMS, CDL2021pNUAA, SLA2020IJoMaMS, SV2017AGT, Sin1974, CDL2021a}. The purpose of the present paper is to find conditions for the uniqueness of best proximity pairs and best approximations in semimetric spaces. In particular, Theorem~\ref{t4.5}, from Section~\ref{sec5} of the paper, provides the necessary and sufficient four-points conditions on a semimetric space \((X, d)\) under which, for any two disjoint proximinal \(A\), \(B \subseteq X\) there is at most one best proximity pair \((a_0, b_0) \in A \times B\). Moreover, in Theorem~\ref{t4.2} of Section~\ref{sec4} we characterize the semimetric spaces whose points have exactly one best approximation in each proximinal subspace.

A more detailed description of the results of the paper will be given in the next section after introduction of the relevant terminology.

\section{Preliminaries}

We will use some concepts from Graph Theory to formulate the results of the paper. For the convenience of the reader, these concepts are recalled below. The section also contains the definitions of some classes of ``rigid'' semimetric spaces related to the uniqueness of the best proximity pairs and best approximations, and the definition of the so-called weak similarities used in the formulation of Theorem~\ref{t4.5}.

A \emph{simple graph} is a pair \((V, E)\) consisting of a nonempty set \(V\) and a set \(E\) whose elements are unordered pairs \(\{u, v\}\) of different elements \(u\), \(v \in V\). For brevity, we will say that \(G\) is a graph if \(G\) is a simple graph.

For a graph \(G = (V, E)\), the sets \(V = V (G)\) and \(E = E(G)\) are called the \emph{set of vertices} and the \emph{set of edges}, respectively. A graph whose edge set is empty is called a \emph{null graph}. Two vertices \(u\), \(v \in V\) are \emph{adjacent} if \(\{u, v\}\) is an edge in \(G\). The \emph{degree} of a vertex \(v_0\) in a graph \(G\), denoted \(\deg(v_0) = \deg_{G} (v_0)\), is the number of all vertices which are adjacent with \(v_0\) in \(G\). A graph in which each pair of distinct vertices are adjacent is a \emph{complete graph}. We will denote the complete graph with a vertex set \(X\) by \(K_{|X|}\) (cf.~\cite[p.~17]{Wil1996}).

A graph \(H\) is, by definition, a \emph{subgraph} of a graph \(G\) if the inclusions \(V (H) \subseteq V (G) \) and \(E(H) \subseteq E(G)\) hold.

A graph \(G\) is \emph{finite} if \(V (G)\) is a finite set, \(|V (G)| < \infty\). We will consider graphs having the vertex sets of arbitrary cardinality.

If \(\{G_i \colon i \in I\}\) is a nonempty family of graphs, then the \emph{union} of the graphs \(G_i\), \(i \in I\), is a graph \(G^*\) such that
\[
V(G^*) = \bigcup_{i \in I} V(G_i) \quad \text{and} \quad E(G^*) = \bigcup_{i \in I} E(G_i).
\]
The union \(G^*\) is \emph{disjoint} if \(V(G_{i_1}) \cap V(G_{i_2}) = \varnothing\) holds for all different \(i_1\), \(i_2 \in I\).

\begin{definition}\label{d1.4}
A graph \(G\) is \emph{bipartite} if the vertex set \(V(G)\) can be partitioned into two nonvoid disjoint sets, or \emph{parts}, in such a way that no edge has both ends in the same part. A bipartite graph in which every two vertices from different parts are adjacent is called \emph{complete bipartite}.
\end{definition}

An important subclass of complete bipartite graphs is formed by the so-called stars. We shall say that a graph \(S\) is a \emph{star} if \(|V(S)| \geqslant 2\) and there is a vertex \(c \in V(S)\), the \emph{center} of \(S\), such that \(S\) is complete bipartite with the parts \(\{c\}\) and \(V(S) \setminus \{c\}\). We will use the concept of stars at the end of Section~\ref{sec4} of the paper.

\begin{definition}\label{d1.5}
A bipartite graph \(G\) with fixed parts \(A\) and \(B\) is \emph{proximinal} if there exists a semimetric space \((X, d)\) such that \(A\) and \(B\) are disjoint proximinal subsets of \(X\), and the equivalence
\begin{equation}\label{d1.5:e1}
\bigl(\{a, b\} \in E(G)\bigr) \Leftrightarrow \bigl(d(a, b) = \dist(A, B)\bigr)
\end{equation}
is valid for every \(a \in A\) and every \(b \in B\). In this case we write \(G = G_X(A, B) = G_{X, d}(A, B)\) and say that \(G\) is proximinal for \((X, d)\).
\end{definition}

Investigations of proximinal graphs were started in~\cite{CDL2021a}.

Let us recall now the fundamental concept of graph isomorphism.

\begin{definition}\label{d1.6}
Let \(G_1\) and \(G_2\) be simple graphs. A bijection \(f \colon V(G_1) \to V(G_2)\) is an \emph{isomorphism} of \(G_1\) and \(G_2\) if
\[
(\{u, v\} \in E(G_1)) \Leftrightarrow (\{f(u), f(v)\} \in E(G_2))
\]
is valid for all \(u\), \(v \in V(G_1)\). Two graphs are \emph{isomorphic} if there exists an isomorphism of these graphs.
\end{definition}

We also need the notion of \emph{digraph isomorphism}. Following~\cite{CL1996} we shall say that a \emph{digraph} \(D\) is a nonempty set \(V(D)\) of \emph{vertices} together with a (possible empty) set \(E(D)\) of ordered pairs of distinct vertices of \(D\) called \emph{arcs}.

A digraph \(D_1\) is isomorphic to a digraph \(D_2\) if there exists a bijection \(f \colon V(D_1) \to V(D_2)\) such that \((u, v) \in E(D_1)\) if and only if \((f(u), f(v)) \in E(D_2)\).

\begin{definition}\label{d1.9}
Let \((X, d)\) be a finite semimetric space with \(|X| \geqslant 2\). Then we denote by \(Di_{X}\) a digraph with \(V(Di_{X}) = E(K_{|X|})\), where \(K_{|X|}\) is the complete graph with the vertex set \(X\), and such that, for any \(u = \{p, q\} \in E(K_{|X|})\) and \(v = \{l, m\} \in E(K_{|X|})\), the relationship
\[
(u, v) \in E(Di_{X})
\]
holds if and only if \(d(p, q) > d(l, m)\) and, for every \(\{x, y\} \in E(K_{|X|})\), the double inequality
\[
d(p, q) \geqslant d(x, y) \geqslant d(l, m)
\]
implies either \(\{x, y\} = u\) or \(\{x, y\} = v\).
\end{definition}

We will use the notion of isomorphic \(Di_X\) in the case $|X| = 4$ to formulate Theorem~\ref{t4.5}.

\begin{remark}\label{r1.9}
Let \((X, d)\) be a finite semimetric space with $|X| \geqslant 2$. Let us define a partial order \(\preccurlyeq_d\) on the set \(E(K_{|X|})\) such that
\[
\bigl(\{p, q\} \prec_d \{l, m\}\bigr) \Leftrightarrow \bigl(d(p, q) < d(l, m)\bigr).
\]
Then \(Di_X\) is the \emph{Hasse diagram} of the poset \((E(K_{|X|}), {\preccurlyeq}_d)\). The definition of Hasse diagrams can be found, for example, in \cite{Schr2003}, page~7.
\end{remark}

Here is an example to illustrate some of concepts introduced above.

\begin{example}\label{ex1.10}
Let \(X = \{p, q, l, m\}\) be the four-point subset of the complex plane \(\CC\) depicted in Figure~\ref{fig1}. Write \(A = \{p, q\}\), \(B = \{l, m\}\) and let \(d\) be the restriction of the usual Euclidean metric on \(X \times X\).
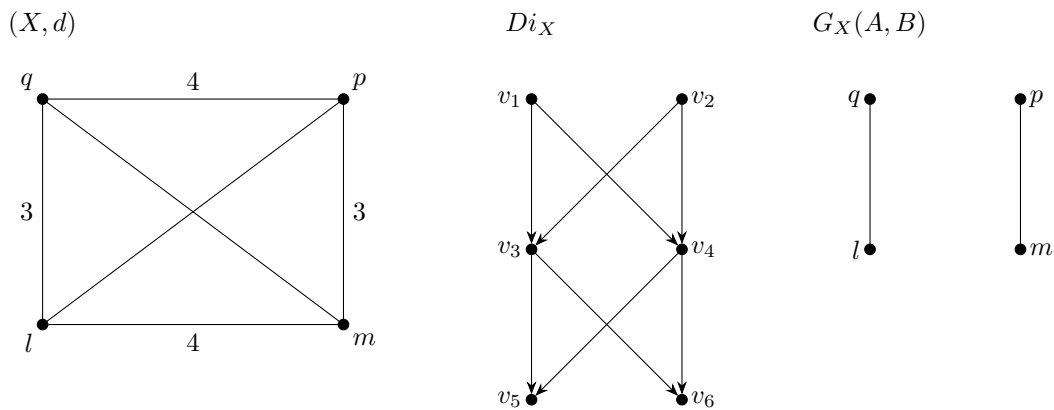
\begin{figure}[ht]
\centering
\begin{tikzpicture}[scale=1,
arrow/.style = {-{Stealth[length=5pt]}, shorten >=2pt}]
\coordinate [label=below left:{$l$}] (A) at (-2, 0);
\coordinate [label=above left:{$q$}] (B) at (-2, 3);
\coordinate [label=above right:{$p$}] (C) at (2, 3);
\coordinate [label=below right:{$m$}] (D) at (2, 0);
\draw [fill, black] (A) circle (2pt);
\draw [fill, black] (B) circle (2pt);
\draw [fill, black] (C) circle (2pt);
\draw [fill, black] (D) circle (2pt);
\draw (A) -- node [left] {\(3\)} (B) -- node [above] {\(4\)} (C) -- node [right] {\(3\)} (D) -- node [below] {\(4\)} (A);
\draw (A) -- (C);
\draw (B) -- (D);
\draw (-2, 3) node[yshift=1cm] {\((X, d)\)};

\begin{scope}[xshift=5.5cm, yshift=3cm]
\def\xx{1cm}
\def\yy{-2cm}
\def\dr{2pt}
\coordinate [label= left:{$v_1$}] (v1) at (-\xx, 0);
\coordinate [label= left:{$v_3$}] (v3) at (-\xx, \yy);
\coordinate [label= left:{$v_5$}] (v5) at (-\xx, 2*\yy);
\coordinate [label= right:{$v_2$}] (v2) at (\xx, 0);
\coordinate [label= right:{$v_4$}] (v4) at (\xx, \yy);
\coordinate [label= right:{$v_6$}] (v6) at (\xx, 2*\yy);
\draw [fill, black] (v1) circle (\dr);
\draw [fill, black] (v2) circle (\dr);
\draw [fill, black] (v3) circle (\dr);
\draw [fill, black] (v4) circle (\dr);
\draw [fill, black] (v5) circle (\dr);
\draw [fill, black] (v6) circle (\dr);
\draw [arrow] (v1) -- (v3);
\draw [arrow] (v1) -- (v4);
\draw [arrow] (v2) -- (v3);
\draw [arrow] (v2) -- (v4);
\draw [arrow] (v3) -- (v5);
\draw [arrow] (v3) -- (v6);
\draw [arrow] (v4) -- (v5);
\draw [arrow] (v4) -- (v6);
\draw (-\xx, 0) node[yshift=\xx] {\(Di_{X}\)};
\end{scope}
\begin{scope}[xshift=10cm, yshift=3cm]
\def\xx{1cm}
\def\yy{-2cm}
\def\dr{2pt}
\coordinate [label= left:{$q$}] (v1) at (-\xx, 0);
\coordinate [label= left:{$l$}] (v3) at (-\xx, \yy);
\coordinate [label= right:{$p$}] (v2) at (\xx, 0);
\coordinate [label= right:{$m$}] (v4) at (\xx, \yy);
\draw [fill, black] (v1) circle (\dr);
\draw [fill, black] (v2) circle (\dr);
\draw [fill, black] (v3) circle (\dr);
\draw [fill, black] (v4) circle (\dr);
\draw (v1) -- (v3);
\draw (v2) -- (v4);
\draw (-\xx, 0) node[yshift=\xx] {\(G_X(A, B)\)};
\end{scope}
\end{tikzpicture}
\caption{The space \((X, d)\), its digraph \(Di_X\), and the proximinal graph \(G_X(A, B)\) for \(A = \{p, q\}\) and \(B = \{l, m\}\).}\label{fig1}
\end{figure}

Then we have:
\begin{itemize}
\item \(\dist (A, B) = 3\);
\item \(V(Di_{X}) = \{v_1, \ldots, v_6\}\),
\end{itemize}
where \(v_1 = \{p, l\}\), \(v_2 = \{q, m\}\), \(v_3 = \{p, q\}\), \(v_4 = \{l, m\}\), \(v_5 = \{m, p\}\), \(v_6 = \{q, l\}\);
\begin{itemize}
\item \(E(Di_{X}) = \bigl\{(v_1, v_3), (v_1, v_4), (v_2, v_3), (v_2, v_4), (v_3, v_5), (v_3, v_5), (v_4, v_5), (v_4, v_6)\bigr\}\);
\item \(V(G_X(A, B)) = X\) and \(E(G_X(A, B)) = \bigl\{\{q, l\}, \{p, m\}\bigr\}\).
\end{itemize}
\end{example}

For every semimetric space \((X, d)\), we denote by \(D(X)\) the set of all nonzero distances between points of the set \(X\),
\[
D(X) = \bigl\{d(x, y) \colon x \neq y \text{ and } x, y \in X\bigr\}.
\]

The following definition is an equivalent form of Definition~1.1 from~\cite{DP2013AMH}.

\begin{definition}\label{d1.11}
Let \((X, d)\) and \((Y, \rho)\) be semimetric spaces with \(|X|\), \(|Y| \geqslant 2\). A mapping \(\Phi \colon X \to Y\) is a \emph{weak similarity} of \((X, d)\) and \((Y, \rho)\) if \(\Phi\) is bijective and there is a bijective strictly increasing function \(\psi \colon D(Y) \to D(X)\) such that the equality
\begin{equation*}
d(x, y) = \psi\left(\rho\bigl(\Phi(x), \Phi(y)\bigr)\right)
\end{equation*}
holds for all \(x\), \(y \in X\).

We say that two semimetric spaces are \emph{weakly similar} is there is a weak similarity of these spaces.
\end{definition}

Definition~\ref{d1.11} generalizes the concept of the similarity of spaces. Let \((X, d)\) and \((Y, \rho)\) be semimetric spaces. A bijective mapping \(\Phi \colon X \to Y\) is a \emph{similarity}, if there is a strictly positive number \(r\), the \emph{ratio} of \(\Phi\), such that
\[
\rho\bigl(\Phi(x), \Phi(y)\bigr) = rd(x, y)
\]
for all \(x\), \(y \in X\) (see, for example, \cite[p.~45]{Edg1992} for metric case). We will say that \((X, d)\) and \((Y, \rho)\) are \emph{isometric} and \(\Phi \colon X \to Y\) is an \emph{isometry} of \((X, d)\) and \((Y, \rho)\) if \(\Phi\) is a similarity with the ratio \(r = 1\).

Some questions connected with the weak similarities were studied in \cite{DovBBMSSS2020, DLAMH2020, Dov2019IEJA, BDS2021pNUAA}. The weak similarities of finite ultrametric and semimetric spaces were considered by E.~Petrov in~\cite{Pet2018pNUAA}.
\medskip

Let us introduce now the classes of semimetric spaces which will be used future.

\begin{definition}\label{d1.0}
A semimetric space \((X, d)\) is said to be \emph{strongly rigid} if \(d(x, y) = d(u, v)\) and \(x \neq y\) imply \(\{x, y\} = \{u, v\}\) for all \(x\), \(y\), \(u\), \(v \in X\). We will denote by \(\mathbf{SR}\) (Strong Rigidity) the class of all strongly rigid semimetric spaces.
\end{definition}

Some properties of strongly rigid metric spaces are described in \cite{Martin1977, DLAMH2020, BDKP2017AASFM, Janos1972}.

\begin{definition}\label{d3.1}
A semimetric space \((X, d)\) is \emph{weakly rigid} if every three-point subspace of \((X, d)\) is strongly rigid. We will denote by \(\mathbf{WR}\) (Weak Rigidity) the class of all weakly rigid semimetric spaces.
\end{definition}

\begin{definition}\label{d2.11}
A semimetric space \((X, d)\) belongs to the class \(\mathbf{UBPP}\) (Unique Best Proximity Pair) if for any two disjoint proximinal subsets \(A\) and \(B\) of \(X\) there is at most one best proximity pair \((a_0, b_0) \in A \times B\). We will say that \((X, d)\) is an \(\mathbf{UBPP}\)-space if \((X, d) \in \mathbf{UBPP}\).
\end{definition}

To the best of our knowledge, the classes \(\mathbf{WR}\) and \(\mathbf{UBPP}\) have not previously been considered either in the Fixed Point Theory or in the Distance Geometry.

The results of the paper are presented as follows.

A characterization of graphs which are proximinal for \(\mathbf{SR}\)-spaces and \(\mathbf{UBPP}\)-spaces is obtained in Theorem~\ref{t3.2}. In Example~\ref{ex2.5} we construct a semimetric space belonging to \(\mathbf{UBPP} \setminus \mathbf{SR}\). The proximinal graphs of \(\mathbf{WR}\)-spaces are characterized in Theorem~\ref{t4.1}. Theorem~\ref{t4.2} shows, in particular, that a semimetric space \((X, d)\) is weakly rigid iff every \(x \in X\) has exactly one best approximation in every proximinal subspace of \((X, d)\). A four-point \((X, d) \in \mathbf{WR} \setminus \mathbf{UBPP}\) is constructed in Example~\ref{ex3.4}. An interesting interrelation between proximinal subsets of weakly rigid semimetric spaces and ultrametric spaces is described in Proposition~\ref{p3.6}. In Lemma~\ref{l4.5} we prove that every four-point space \((Y, \rho) \in \mathbf{WR} \setminus \mathbf{UBPP}\) is weakly similar to the semimetric space from Example~\ref{ex3.4}, and this is the most technically difficult result of the paper. The final result of the paper, Theorem~\ref{t4.5}, characterizes \(\mathbf{UBPP}\)-spaces by a set of four-point conditions.

\section{Proximinal graphs for strongly rigid spaces}

Let us recall a characterization of proximinal graphs.

\begin{theorem}[\cite{CDL2021a}]\label{t3.1}
Let \(G\) be a bipartite graph with fixed parts \(A\) and \(B\). Then following statements \ref{t3.1:s1}--\ref{t3.1:s3} are equivalent.
\begin{enumerate}
\item\label{t3.1:s1} \(G\) is proximinal for a metric space.
\item\label{t3.1:s2} \(G\) is proximinal for a semimetric space.
\item\label{t3.1:s3} Either \(G\) is not a null graph or \(G\) is a null graph but \(A\) and \(B\) are infinite.
\end{enumerate}
\end{theorem}

The next theorem can be considered as a refinement of Theorem~\ref{t3.1} for the case of strongly rigid spaces.

\begin{theorem}\label{t3.2}
Let \(G\) be a bipartite graph with fixed parts \(A\) and \(B\). Then following statements \(\ref{t3.2:s1}\)--\(\ref{t3.2:s4}\) are equivalent.
\begin{enumerate}
\item \label{t3.2:s1} \(G\) is proximinal for a strongly rigid metric space.
\item \label{t3.2:s2} \(G\) is proximinal for a strongly rigid semimetric space.
\item \label{t3.2:s3} \(G\) is proximinal for an \(\mathbf{UBPP}\)-space.
\item \label{t3.2:s4} The following conditions are simultaneously fulfilled:
\begin{enumerate}
\item \label{t3.2:s4.1} The inequalities \(|E(G)| \leqslant 1\) and \(|V(G)| \leqslant \mathfrak{c}\) hold, where \(\mathfrak{c}\) is the cardinality of the continuum.
\item \label{t3.2:s4.2} If \(G\) is a null graph, then \(A\) and \(B\) are infinite.
\end{enumerate}
\end{enumerate}
\end{theorem}

\begin{proof}
\(\ref{t3.2:s1} \Rightarrow \ref{t3.2:s2}\). This implication is trivially valid because every strongly rigid metric is a strongly rigid semimetric.

\(\ref{t3.2:s2} \Rightarrow \ref{t3.2:s3}\). It suffices to show that the inclusion
\begin{equation}\label{t3.2:e10}
\mathbf{SR} \subseteq \mathbf{UBPP}
\end{equation}
holds.

Let us consider a semimetric space \((X, d) \notin \mathbf{UBPP}\). Then, by Definition~\ref{d1.11}, there exist disjoint proximinal sets \(A\), \(B \subseteq X\) and \((x, y)\), \((u, v) \in A \times B\) such that \(\{x, y\} \neq \{u, v\}\) and
\[
d(x, y) = d(u, v) = \dist(A, B).
\]
Hence, we have \((X, d) \notin \mathbf{SR}\) by Definition~\ref{d1.0}. Inclusion~\eqref{t3.2:e10} follows.

\(\ref{t3.2:s3} \Rightarrow \ref{t3.2:s4}\). Suppose that there is \((X, d) \in \mathbf{UBPP}\) such that \(A\) and \(B\) are disjoint proximinal subsets of \(X\) and \(G = G_{X}(A, B)\) holds. To prove \ref{t3.2:s4}, we first note that the inequality \(|E(G)| \leqslant 1\) follows from Definition~\ref{d2.11}. Moreover, if \(v_0\) is a given point of \(X\), then the mapping
\[
X \setminus \{v_0\} \ni u \mapsto d(u, v_0) \in [0, \infty)
\]
is injective. Indeed, if \(u_1\) and \(u_2\) are different points of \(X \setminus \{v_0\}\), then the sets \(\{v_0\}\) and \(\{u_1, u_2\}\) are proximinal subspaces of \((X, d)\) and, consequently, \(d(u_1, v_0) \neq d(u_2, v_0)\) follows from \((X, d) \in \mathbf{UBPP}\) by Definition~\ref{d1.11}. Hence, we have
\begin{equation}\label{t3.2:e2}
|V(K_{|X|})| \leqslant \card([0, \infty)) + 1 = \mathfrak{c} + 1 = \mathfrak{c}
\end{equation}
(see Definition~\ref{d1.9}). Since \(G\) is a subgraph of \(K_{|X|}\), the inequality \(|V(G)| \leqslant |V(K_{|X|})|\) holds. The last inequality, the inequality \(|E(G)| \leqslant 1\) and \eqref{t3.2:e2} imply \ref{t3.2:s4.1}.

Condition~\ref{t3.2:s4.2} follows from statement~\ref{t3.1:s3} of Theorem~\ref{t3.1}.

\(\ref{t3.2:s4} \Rightarrow \ref{t3.2:s1}\). Let \ref{t3.2:s4} hold. Write \(X = A \cup B\). Our goal is to construct a strongly rigid metric \(d \colon X \times X \to [0, \infty)\) such that
\begin{equation}\label{t3.2:e2.1}
G = G_{X, d}(A, B).
\end{equation}

First of all, we note that \ref{t3.2:s4.1} implies
\begin{equation}\label{t3.2:e3}
|A| \leqslant \mathfrak{c} \quad \text{and} \quad  |B| \leqslant \mathfrak{c}.
\end{equation}

Suppose that \(G\) is a null graph. Let us consider the case when \(|A| = |B|\). Then \(A\) and \(B\) are infinite by~\ref{t3.2:s4.2} and, consequently, from \(|A| = |B|\) it follows that there is a bijective mapping \(\Phi \colon A \to B\). Write
\begin{equation}\label{t3.2:e4}
D:= \left\{1 + \frac{1}{2n} \colon n \in \NN\right\},
\end{equation}
where \(\NN\) is the set of all positive integer numbers. Since \(D\) is countably infinite and \(A\) is infinite, there is a subset \(A^1\) of \(A\) such that \(|A^1| = |D|\). Let us define the subsets \(E^1\), \(E^2\) and \(E^3\) of the set \(E(K_{|X|})\) by
\begin{equation}\label{t3.2:e4.1}
E^1 := \bigl\{\{a, \Phi(a)\} \colon a \in A^1\bigr\}, \quad
E^2 := \begin{cases}
\bigl\{\{a, \Phi(a)\} \colon a \in A \setminus A^1\bigr\} & \text{if } A^1 \neq A,\\
\varnothing & \text{if } A^1 = A,
\end{cases}
\end{equation}
and
\begin{equation}\label{t3.2:e4.2}
E^3 := E(K_{|X|}) \setminus (E^1 \cup E^2).
\end{equation}
It is clear that \(E^1\), \(E^2\) and \(E^3\) are disjoint and \(E(K_{|X|}) = E^1 \cup E^2 \cup E^3\) holds. Hence, there is an injective mapping \(f \colon E(K_{|X|}) \to (1, 2]\) such that the restriction \(f|_{E^1}\) is a bijection of \(E^1\) on \(D\), and
\begin{equation}\label{t3.2:e5}
f(E^2) \subseteq \left(\frac{3}{2}, \frac{7}{4}\right], \quad \text{and} \quad f(E^3) \subseteq \left(\frac{7}{4}, 2\right].
\end{equation}
Using \eqref{t3.2:e5} and \(0 \notin D\), we see that there is a unique semimetric \(d \colon X \times X \to [0, \infty)\) satisfying
\begin{equation}\label{t3.2:e6}
d(x, y) = f(\{x, y\})
\end{equation}
for every \(\{x, y\} \in E(K)\). We claim that
\begin{itemize}
\item \(d\) is a strongly rigid metric, \(A\) and \(B\) are disjoint proximinal subsets in \((X, d)\) and \eqref{t3.2:e2.1} holds.
\end{itemize}
To prove that \(d\) is a metric on \(X\), we note that
\begin{equation}\label{t3.2:e6.1}
f(E^1) = D \subseteq \left(1, \frac{3}{2}\right]
\end{equation}
holds by \eqref{t3.2:e4}. The equality \(E(K_{|X|}) = E^1 \cup E^2 \cup E^3\) and \eqref{t3.2:e5}--\eqref{t3.2:e6.1} imply the double inequality
\begin{equation}\label{t3.2:e7}
1 < d(x, y) \leqslant 2
\end{equation}
for all different points \(x\) and \(y\) of \(X\). Now the triangle inequality easily follows from~\eqref{t3.2:e7}. Thus, the semimetric \(d\) is a metric. The metric \(d\) is strongly rigid because the mapping \(f \colon E(K_{|X|}) \to (1, 2]\) is injective.

Let us prove that \(A\) and \(B\) are proximinal subsets of \((X, d)\).

For every \(x \in A = X \setminus B\), \eqref{t3.2:e5} and \eqref{t3.2:e6.1} imply that we have
\[
d(x, \Phi(x)) \leqslant \frac{3}{2} \quad \text{and} \quad d(x, a) > \frac{3}{2}
\]
whenever \(a \in A\) and \(a \neq \Phi(x)\). Thus, \(\Phi(x)\) is the unique best approximation to \(x \in A\) in \(B\) and, consequently, \(B\) is a proximinal set in \((X, d)\). Similarly, using the inverse mapping \(\Phi^{-1} \colon B \to A\) instead of \(\Phi\), we see that
\[
d(y, \Phi^{-1}(y)) = \dist (y, A)
\]
holds for every \(y \in B = X \setminus A\). Thus, \(A\) is also a proximinal set in \((X, d)\).

The equality \(f(E^1) = D\), \eqref{t3.2:e4} and \eqref{t3.2:e7} imply that
\[
\dist(A, B) = \lim_{n \to \infty} \left(1 + \frac{1}{2n}\right) = 1.
\]
Hence, \(G_{X}(A, B)\) is a null graph by~\eqref{t3.2:e7} and Definition~\ref{d1.5}. Consequently, we have the equalities
\[
V(G) = V(G_X(A, B)) \quad \text{and} \quad E(G) = E(G_X(A, B)) = \varnothing,
\]
i.e. \eqref{t3.2:e2.1} holds.

Let us consider now the case when \(G\) is a null graph and \(|A| \neq |B|\). Assume, without loss of generality, that \(|A| < |B|\). Then there is an injective mapping \(\Phi^1 \colon A \to B\). Write
\[
D^1 := \left\{1 + \frac{1}{4n} \colon n \in \NN\right\}
\]
(cf. \eqref{t3.2:e4}). Since \(D^1\) is countably infinite and, by condition~\ref{t3.2:s4.2}, \(A\) is infinite, there is a subset \(A^{1,1}\) of \(A^1\) such that \(|D^1| = |A^{1,1}|\). Let \(a^1\) be a given point of \(A\). Write
\begin{equation*}
\begin{aligned}
E^{1,1} & := \bigl\{\{a, \Phi^1(a)\} \colon a \in A^{1,1}\bigr\}, \\
E^{2,1} & := \begin{cases}
\bigl\{\{a, \Phi^1(a)\} \colon a \in A \setminus A^{1,1}\bigr\} & \text{if } A^{1,1} \neq A\\
\varnothing & \text{if } A^{1,1} = A,
\end{cases}\\
E^{2,2} & := \bigl\{\{a^1, b\} \colon b \in B \setminus \Phi^1(A^1)\bigr\},\\
E^{3,1} & := E(K) \setminus (E^{1,1} \cup E^{2,1} \cup E^{2,2}).
\end{aligned}
\end{equation*}
It is clear that \(E^{1,1}\), \(E^{2,1}\), \(E^{2,2}\) and \(E^{3,1}\) are disjoint subsets of \(E(K_{|X|})\) and
\[
E(K) = E^{1,1} \cup E^{2,1} \cup E^{2,2} \cup E^{3,1}.
\]

Let us consider an injective mapping \(f^1 \colon E(K_{|X|}) \to (1, 2]\) such that the restriction \(f^1|_{E^{1,1}}\) is a bijection of \(E^{1,1}\) on \(D^1\) and
\[
f^1(E^{2,1}) \subseteq \left(\frac{5}{4}, \frac{6}{4}\right], \quad f^1(E^{2,2}) \subseteq \left(\frac{6}{4}, \frac{7}{4}\right], \quad f^1(E^{3,1}) \subseteq \left(\frac{7}{4}, 2\right]
\]
hold. Then there is a unique semimetric \(d^1 \colon X \times X \to [0, \infty)\) satisfying the equality
\[
d^1(x, y) = f(\{x, y\})
\]
for every \(\{x, y\} \in E(K_{|X|})\). Arguing similarly to the case \(|A| = |B|\), we can show that
\begin{itemize}
\item \((X, d^1)\) is a strongly rigid metric space, \(A\) and \(B\) are disjoint proximinal subsets in \((X, d^1)\) and \eqref{t3.2:e2.1} holds with \(d = d^1\).
\end{itemize}
Thus, the implication \(\ref{t3.2:s4} \Rightarrow \ref{t3.2:s1}\) is valid if \(G\) is a null graph.

Let us consider now the case \(E(G) \neq \varnothing\). Then, by condition~\ref{t3.2:s4.1}, there is the unique pair \((a_0, b_0)\) of points such that \(a_0 \in A\), \(b_0 \in B\) and \(\{a_0, b_0\} \in E(G)\). Let us define the subsets \(E^{*,0,0}\), \(E^{*,0,1}\), \(E^{*,1,0}\) and \(E^{*,1,1}\) of \(E(K_{|X|})\) by
\begin{equation}\label{t3.2:e8}
\begin{aligned}
E^{*,0,0} & := \bigl\{\{a_0, b_0\}\bigr\}, \\
E^{*,0,1} & := \begin{cases}
\bigl\{\{a_0, b\} \colon b \in B \setminus \{b_0\}\bigr\} & \text{if } B \neq \{b_0\}\\
\varnothing & \text{if } B = \{b_0\},
\end{cases}\\
E^{*,1,0} & := \begin{cases}
\bigl\{\{a, b_0\} \colon a \in A \setminus \{a_0\}\bigr\}& \text{if } A \neq \{a_0\}\\
\varnothing & \text{if } A = \{a_0\},
\end{cases}\\
E^{*,1,1} & := E(K) \setminus (E^{*,0,0} \cup E^{*,0,1} \cup E^{*,1,0}).
\end{aligned}
\end{equation}
These subsets of \(E(K_{|X|})\) are disjoint and \(E(K_{|X|}) = E^{*,0,0} \cup E^{*,0,1} \cup E^{*,1,0} \cup E^{*,1,1}\) holds. Let us consider an injective mapping \(f^0 \colon E(K_{|X|}) \to [1,2]\) such that
\begin{equation}\label{t3.2:e9}
\begin{aligned}
f^{0}(\{a_0, b_0\}) & = 1, &
f^{0}(E^{*,0,1}) & \subseteq \left(1, \frac{3}{2}\right), \\
f^{0}(E^{*,1,0}) & \subseteq \left[\frac{3}{2}, \frac{7}{4}\right), &
f^{0}(E^{*,1,1}) &\subseteq \left(\frac{7}{4}, 2\right].
\end{aligned}
\end{equation}
Then there is the unique semimetric \(d^{0} \colon X \times X \to [0, \infty)\) satisfying the equality
\[
d^{0}(x, y) = f(\{x, y\})
\]
for every \(\{x, y\} \in E(K_{|X|})\). As in the case \(E(G) = \varnothing\), we can prove that \((X, d^{0})\) is a strongly rigid metric space. Moreover, using \eqref{t3.2:e8} and \eqref{t3.2:e9}, we can show that the equality
\[
\dist(A, B) = d^{0}(a_0, b_0)
\]
holds and, for every \(a \in A\) and \(b \in B\), we have
\[
\dist(a, B) = d^{0}(a, b_0), \quad \dist(b, A) = d^{0}(a_0, b) \quad \text{and} \quad d(a, b) > \dist(A, B)
\]
whenever \(a \in A \setminus \{a_0\}\) and \(b \in B \setminus \{b_0\}\). Thus, \(A\) and \(B\) are proximinal in \((X, d^{0})\) and \((a_0, b_0) \in A \times B\) is the unique best proximity pair for \(A\) and \(B\) (see Definition~\ref{d1.2}). Using Definition~\ref{d1.5}, we obtain the equalities
\[
V(G) = V(G_{X, d^{0}}(A, B)) \quad \text{and} \quad E(G) = E(G_{X, d^{0}}(A, B)).
\]
Equality \eqref{t3.2:e2.1} follows with \(d = d^0\).
\end{proof}

Theorem~\ref{t3.2} implies the following.

\begin{corollary}\label{c2.6}
A graph \(G\) is not isomorphic to any proximinal graph for any strongly rigid metric space (\(\mathbf{UBPP}\)-space) if and only if \(G\) is a finite null graph or satisfies at least one from the inequalities \(|V(G)| > \mathfrak{c}\) and \(|E(G)| > 1\).
\end{corollary}

The following proposition shows, in particular, that, for ``small'' semimetric spaces \((X, d)\), the conditions \((X, d) \in \mathbf{UBPP}\) and \((X, d) \in \mathbf{SR}\) are equivalent.

\begin{proposition}\label{p2.4}
Let \((X, d)\) be a semimetric space with \(|X| \leqslant 3\). Then the following conditions are equivalent:
\begin{enumerate}
\item \label{p2.4:s1} \((X, d) \in \mathbf{SR}\).
\item \label{p2.4:s2} \((X, d) \in \mathbf{WR}\).
\item \label{p2.4:s3} \((X, d) \in \mathbf{UBPP}\).
\item \label{p2.4:s4} Each proximinal graph \(G_{X, d}(A, B)\) has exactly one edge.
\end{enumerate}
\end{proposition}

\begin{proof}
It is clear from Definition~\ref{d1.5} that the equivalence \(\ref{p2.4:s3} \Leftrightarrow \ref{p2.4:s4}\) is valid. Similarly, Definition~\ref{d3.1} implies the validity if \(\ref{p2.4:s1} \Leftrightarrow \ref{p2.4:s2}\). Moreover, inclusion~\eqref{t3.2:e10} shows that the implication \(\ref{p2.4:s1} \Rightarrow \ref{p2.4:s3}\) is also valid. Thus, to complete the proof it enough to check that \ref{p2.4:s4} implies \ref{p2.4:s1}.

Suppose contrary that \(\ref{p2.4:s4}\) holds but \((X, d) \notin \mathbf{SR}\). Since \(|X| \leqslant 2\) implies that \((X, d)\) is strongly rigid, we obtain \(|X| = 3\). Write \(X = \{x_1, x_2, x_3\}\). Since \((X, d)\) is not strongly rigid, we may assume, without loss of generality, that
\begin{equation}\label{p2.4:e1}
d(x_1, x_2) = d(x_2, x_3).
\end{equation}
Let us consider the proximinal graph \(G_{X, d}(A, B)\) with \(A = \{x_1, x_3\}\) and \(B = \{x_2\}\). Equality \eqref{p2.4:e1} and \eqref{d1.5:e1} imply that \(\{x_1, x_2\}\) and \(\{x_2, x_3\}\) are edges of \(G_{X, d}(A, B)\), contrary to \ref{p2.4:s4}.
\end{proof}

The following example shows that the inequality \(|X| \leqslant 3\) cannot be replaced by \(|X| \leqslant 4\) in Proposition~\ref{p2.4}.

\begin{figure}[h]
\centering
\begin{tikzpicture}[scale=1,
arrow/.style = {-{Stealth[length=5pt]}, shorten >=2pt}]
\def\xx{0.5cm}

\coordinate [label=below left:{$-2$}] (A) at (-2*\xx, 0);
\coordinate [label=above right:{$3$}] (B) at (0, 3*\xx);
\coordinate [label=below right:{$5$}] (C) at (5*\xx, 0);
\coordinate [label=below right:{$-4$}] (D) at (0, -4*\xx);
\draw [->] (-2.5*\xx, 0) -- (5.5*\xx,0) node [label=above:{$x$}] {};
\draw [->] (0, -5*\xx) -- (0,4*\xx) node [label=left:{$y$}] {};
\draw [fill, black] (A) circle (2pt);
\draw [fill, black] (B) circle (2pt);
\draw [fill, black] (C) circle (2pt);
\draw [fill, black] (D) circle (2pt);
\draw (A) -- (B) -- (C) -- (D) -- (A);
\draw (-3*\xx, 3*\xx) node [left] {\((X, d)\)};

\begin{scope}[xshift=7cm, yshift=1cm]
\def\xx{1cm}
\def\yy{-1cm}
\def\dr{2pt}

\coordinate (v1) at (-\xx, 0);
\coordinate (v2) at (\xx, 0);
\coordinate (v3) at (0, \yy);
\coordinate (v4) at (0, 2*\yy);
\coordinate (v5) at (0, 3*\yy);
\coordinate (v6) at (0, 4*\yy);
\draw [fill, black] (v1) circle (\dr);
\draw [fill, black] (v2) circle (\dr);
\draw [fill, black] (v3) circle (\dr);
\draw [fill, black] (v4) circle (\dr);
\draw [fill, black] (v5) circle (\dr);
\draw [fill, black] (v6) circle (\dr);
\draw [arrow] (v1) -- (v3);
\draw [arrow] (v2) -- (v3);
\draw [arrow] (v3) -- (v4);
\draw [arrow] (v4) -- (v5);
\draw [arrow] (v5) -- (v6);
\draw (0, 0.5) node[above] {\(Di_X\)};
\end{scope}
\end{tikzpicture}
\caption{The quadrilateral \((X, d) \in \mathbf{UBPP} \setminus \mathbf{SR}\) and its digraph \(Di_X\).}\label{fig2}
\end{figure}
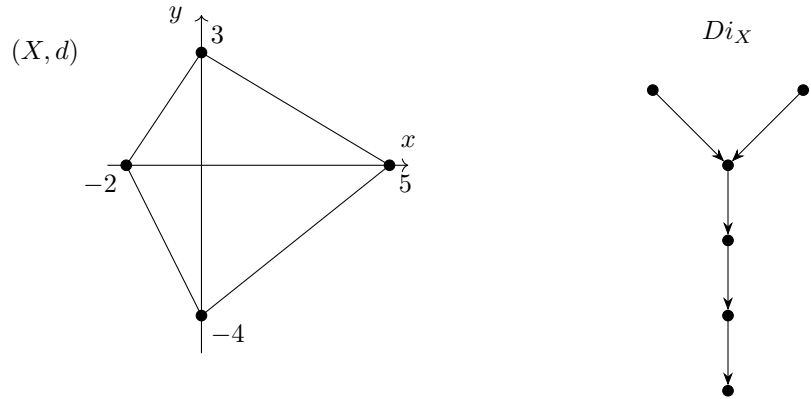

\begin{example}\label{ex2.5}
Let \(X = \{z_1, z_2, z_3, z_4\}\) be the four-point subset of the complex plane \(\CC\) with
\[
z_1 = 5 + 0i, \quad z_2 = 0 + 3i, \quad z_3 = -2 + 0i, \quad z_4 = 0 - 4i
\]
and let \(d\) be the restriction of the usual Euclidean metric on \(X \times X\). It can be proved directly that \((X, d)\) belongs to \(\mathbf{UBPP}\). The diagonals of the quadrilateral with the vertices \(z_1\), \(\ldots\), \(z_4\) are equal to \(7\). Thus, we have \((X, d) \notin \mathbf{SR}\). Furthermore, Pythagorean theorem implies that all sides of \(X\) are strictly less than \(7\) and have pairwise different lengths. Therefore, the digraph \(Di_X\) can be represented as in Figure~\ref{fig2}.
\end{example}

\begin{remark}\label{r3.6}
Lemma~\ref{l3.7} from the last section of the paper implies that each four-point semimetric space \(Y\) belongs to \(\mathbf{UBPP}\) whenever \(Di_Y\) is isomorphic to \(Di_X\) from Figure~\ref{fig2}.
\end{remark}

\section{Uniqueness of the best approximation}
\label{sec4}

In this section we will obtain characterizations of \(\mathbf{WR}\)-spaces and describe the structure of proximinal graphs of such spaces.

\begin{theorem}\label{t4.1}
Let \(G\) be a bipartite graph with fixed parts \(A\) and \(B\). Then following statements \(\ref{t3.2:s1}\)--\(\ref{t3.2:s3}\) are equivalent.
\begin{enumerate}
\item \label{t4.1:s1} \(G\) is proximinal for a weakly rigid metric space.
\item \label{t4.1:s2} \(G\) is proximinal for a weakly rigid semimetric space.
\item \label{t4.1:s3} The following conditions are simultaneously fulfilled:
\begin{enumerate}
\item \label{t4.1:s3.1} The inequality \(|V(G)| \leqslant \mathfrak{c}\) holds, where \(\mathfrak{c}\) is the cardinality of the continuum and we have the inequality \(\deg_{G}(v) \leqslant 1\) for every \(v \in V(G)\).
\item \label{t4.1:s3.2} If \(G\) is a null graph, then \(A\) and \(B\) are infinite.
\end{enumerate}
\end{enumerate}
\end{theorem}

\begin{proof}
\(\ref{t4.1:s1} \Rightarrow \ref{t4.1:s2}\). This implication is trivially valid because every weakly rigid metric is a weakly rigid semimetric.

\(\ref{t4.1:s2} \Rightarrow \ref{t4.1:s3}\). Let \(A\) and \(B\) be disjoint proximinal subsets of \((X, d) \in \mathbf{WR}\) and let
\begin{equation}\label{t4.1:e1}
G = G_{X, d}(A, B).
\end{equation}
Let us prove condition~\ref{t4.1:s3}. Write \(v_0\) for a vertex of \(G\). As in the proof of Theorem~\ref{t3.2}, it is easy to verify that a mapping
\begin{equation}\label{t4.1:e2}
X \setminus \{v_0\} \ni u \mapsto d(u, v_0) \in [0, \infty)
\end{equation}
is injective, that implies the inequality  \(|V(G)| \leqslant \mathfrak{c}\). The injectivity of \eqref{t4.1:e2} also implies that there is at most one \(u \in V(G)\) such that
\[
d(v_0, u) = \dist(A, B).
\]
Consequently, we have \(\deg_{G}(v_0) \leqslant 1\) by \eqref{d1.5:e1}. Thus, condition~\ref{t4.1:s3.1} holds.

Condition~\ref{t4.1:s3.2} follows from statement~\ref{t3.1:s3} of Theorem~\ref{t3.1}.

\(\ref{t4.1:s3} \Rightarrow \ref{t4.1:s1}\). Let statement~\ref{t4.1:s3} hold. We must construct a weakly rigid metric space \((X, d)\) such that \eqref{t4.1:e1} is valid for some disjoint proximinal subsets \(A\) and \(B\) of \(X\).

If \(G\) is a null graph, then, by Theorem~\ref{t3.2}, there exist a strongly rigid metric space \((X, d)\) and disjoint proximinal \(A\), \(B \subseteq X\) such that \eqref{t4.1:e1} holds. Since every strongly rigid metric is weakly rigid, \ref{t4.1:s1} follows from \ref{t4.1:s3} when \(G\) is a null graph.

Let us consider the case \(E(G) \neq \varnothing\). Write
\begin{equation}\label{t4.1:e4}
A_0 := \{v \in A \colon \deg v = 1\} \quad \text{and} \quad
B_0 := \{v \in B \colon \deg v = 1\}.
\end{equation}

Let \(a_0^* \in A\) and \(b_0^* \in B\) be some adjacent vertices of \(G\). Let us consider the following sets
\begin{equation}\label{t4.1:e6}
\begin{aligned}
S_{1,1} &:= \bigl\{(a, b) \in A \times B \colon \{a, b\} \in E(G)\bigr\}, &
S_{1,0} &:= \bigl\{(a_0^*, b) \in A \times B \colon b \in B \setminus B_0\bigr\},\\
S_{0,1} &:= \bigl\{(a, b_0^*) \in A \times B \colon a \in A \setminus A_0\bigr\}, &
\widetilde{S} &:= (A \times B) \setminus (S_{1,1} \cup S_{1,0} \cup S_{0,1}).
\end{aligned}
\end{equation}
It is clear that \(S_{1,1}\), \(S_{1,0}\), \(S_{0,1}\), \(\widetilde{S}\) are disjoint and the equality
\begin{equation}\label{t4.1:e3}
A \times B = S_{1,1} \cup S_{1, 0} \cup S_{0, 1} \cup \widetilde{S}
\end{equation}
holds. Condition~\ref{t4.1:s3.2} implies the inequality \(|S| \leqslant \mathfrak{c}\) for every \(S \subseteq A \times B\). Hence, there is a function \(g \colon A \times B \to [0, \infty)\) such that \(g|_{(A \times B) \setminus S_{1,1}}\) is injective and
\begin{align}\label{t4.1:e7}
g(s) = 1 &\quad \text{for every} \quad s \in S_{1,1};\\
\label{t4.1:e8}
1 < g(s) \leqslant \frac{6}{5} &\quad \text{for every} \quad s \in S_{1,0};\\
\label{t4.1:e9}
\frac{6}{5} < g(s) \leqslant \frac{7}{5} &\quad \text{for every} \quad s \in S_{0,1};\\
\label{t4.1:e10}
\frac{7}{5} < g(s) \leqslant \frac{8}{5} &\quad \text{for every} \quad s \in \widetilde{S}.
\end{align}
Let \(d_A \colon A \times A \to [0, \infty)\) and \(d_B \colon B \times B \to [0, \infty)\) be strongly rigid metrics for which the inclusions
\begin{equation}\label{t4.1:e11}
d_A(A \times A) \subseteq\left(\frac{8}{5}, \frac{9}{5}\right] \quad \text{and} \quad d_B(B \times B) \subseteq\left(\frac{9}{5}, 2\right]
\end{equation}
hold. Write \(X = A \cup B\). Then the set \(X \times X\) is the union of the disjoint sets \(A \times B\), \(B \times A\), \(A \times A\) and \(B \times B\). Moreover, the set \(A \times B\) also is the union of the disjoint sets \(S_{1,1}\), \(S_{1,0}\), \(S_{0,1}\), \(\widetilde{S}\). Therefore, one can construct a unique symmetric mapping \(d\) on \(X \times X\) such that
\begin{equation}\label{t4.1:e12}
d|_{A \times A} = d_A, \quad d|_{B \times B} = d_B \quad \text{and} \quad d|_{A \times B} = g.
\end{equation}
Since \(d_A\) and \(d_B\) are metrics and we have either \((x, x) \in A \times A\) or \((x, x) \in B \times B\) for every \(x \in X\), the equivalence
\[
(d(x, y) = 0) \Leftrightarrow (x = y)
\]
is valid for all \(x\), \(y \in X\). Hence, \(d\) is a semimetric on \(X\). Furthermore, from \eqref{t4.1:e7}--\eqref{t4.1:e11} follows the double inequality \(1 \leqslant d(x, y) \leqslant 2\) whenever \(d(x, y) \neq 0\). Hence, \(d\) is a metric on \(X\).

We claim that \(d\) is weakly rigid. Indeed, on the contrary, suppose that \(d \notin \mathbf{WR}\). Then, using Definitions~\ref{d1.0} and \ref{d3.1}, we can find different points \(x_1\), \(x_2\), \(x_3 \in X\) such that
\begin{equation}\label{t4.1:e13}
d(x_2, x_1) = d(x_2, x_3) \neq 0.
\end{equation}
Without loss of generality, we may assume \(x_2 \in A\). (The case \(x_2 \in B\) is completely similar.) The metric \(d_A\) is strongly rigid. Hence, from~\eqref{t4.1:e13} it follows that at least one from the points \(x_1\), \(x_3\) belongs to the set \(B\). Let us consider first the case when \(x_1 \in B\) and \(x_3 \in B\). Then, using the equality \(d|_{A \times B} = g\), we may rewrite \eqref{t4.1:e13} in the form
\begin{equation}\label{t4.1:e14}
g(x_2, x_1) = g(x_2, x_3) \neq 0.
\end{equation}
Since the restriction \(g|_{(A \times B) \setminus S_{1,1}}\) is injective, \(S_{1, 1} \cap \{(x_2, x_1), (x_2, x_3)\} \neq \varnothing\) holds. Without loss of generality, we may assume that \((x_2, x_1) \in S_{1,1}\). If \((x_2, x_3) \notin S_{1,1}\), then from~\eqref{t4.1:e7}--\eqref{t4.1:e10} follows the inequality \(g(x_2, x_3) > g(x_2, x_1)\), contrary to~\eqref{t4.1:e14}. Hence, we have
\begin{equation}\label{t4.1:e15}
(x_2, x_1) \in S_{1,1} \quad \text{and} \quad (x_2, x_3) \in S_{1,1}.
\end{equation}
Memberships \eqref{t4.1:e15}, and definitions~\eqref{t4.1:e6}, and \(x_1 \neq x_2\) imply the inequality \(\deg_{G}(x_2) \geqslant 2\), that contradicts condition~\ref{t4.1:s3.2}. Thus, we have either \(x_1 \in B\) and \(x_3 \in A\) or \(x_1 \in A\) and \(x_3 \in B\). If \(x_1 \in B\) and \(x_3 \in A\), then, using~\eqref{t4.1:e11} and \eqref{t4.1:e7}--\eqref{t4.1:e10}, we obtain the inequality
\begin{equation}\label{t4.1:e18}
d(x_2, x_1) < d(x_2, x_3).
\end{equation}
Analogously, \(x_1 \in A\) and \(x_3 \in B\) imply the inequality
\begin{equation}\label{t4.1:e19}
d(x_2, x_3) < d(x_2, x_1).
\end{equation}
Each of inequalities \eqref{t4.1:e18}--\eqref{t4.1:e19} contradicts \eqref{t4.1:e13}. Thus, \(d\) is a weakly rigid metric on \(X\).

We now prove that \(B\) is proximinal subspace of \((X, d)\).

It follows from Definition~\ref{d1.1} and \(X = A \cup B\) that the set \(B\) is proximinal in \((X, d)\) iff, for every \(x \in A\), there is a best approximation \(b_0 = b_0(x)\) to \(x\) in \(B\), \(d(x, b_0) = \inf\bigl\{d(x, b) \colon b \in B\bigr\}\). Using \eqref{t4.1:e12} we obtain that \(B\) is proximinal iff, for every \(x \in A\), there is \(b_0 = b_0(x) \in B\) such that
\[
g(x, b_0) = \inf\bigl\{g(x, b) \colon b \in B\bigr\}.
\]
Suppose that \(x \in A_0\). Then, by \eqref{t4.1:e4}, the equality \(\deg_{G}(x) = 1\) holds. Consequently, we can find \(y_0 \in B\) such that \(\{x, y_0\} \in E(G)\). Now from \eqref{t4.1:e6} and \eqref{t4.1:e7} follows the equality \(g(x, y_0) = 1\). Using \eqref{t4.1:e7}--\eqref{t4.1:e10} we obtain \(\inf\bigl\{g(x, b) \colon b \in B\bigr\} \geqslant 1\). Hence, the point \(y_0 \in B\) is a best approximation to \(x\) in \(B\). Let us consider the case \(x \in A \setminus A_0\). Then \eqref{t4.1:e6} and \eqref{t4.1:e3} imply the membership \((x, b) \in S_{0, 1} \cup \widetilde{S}\) for every \(b \in B\). If \(b_0^* \in B\) is a point defined as in the third formula of \eqref{t4.1:e6}, then \(b_0^*\) is a unique point \(b\) of \(B\) satisfying the membership \((x, b) \in S_{0, 1}\). Consequently, from \eqref{t4.1:e9} and \eqref{t4.1:e10} follows the double inequality
\[
g(x, b_0^*) \leqslant \frac{7}{5} < g(x, b)
\]
for every \(b \in B \setminus \{b_0^*\}\). Hence, \(b_0^*\) is the best approximation to \(x\) in \(B\). Thus, the set \(B\) is a proximinal subset of \((X, d)\).

Reasoning similarly, we also can prove that \(A\) is proximinal in \((X, d)\).

To complete the proof, it suffices to note that \(E(G) \neq \varnothing\) and, therefore, by the definition of the metric \(d\), \((a, b) \in A \times B\) is a best proximity pair for \(A\) and \(B\) if and only if \(\{a, b\} \in E(G)\).
\end{proof}

Let \(A\) and \(B\) be disjoint proximinal subsets of a semimetric space \((X, d)\). In what follows we denote by \(A_0\) and \(B_0\) the sets defined as
\begin{equation}\label{t4.1:e5}
A_0 := \{a \in A \colon \dist(a, B) = \dist(A, B)\} \quad \text{and} \quad
B_0 := \{b \in B \colon \dist(b, A) = \dist(A, B)\}.
\end{equation}

\begin{theorem}\label{t4.2}
Let \((X, d)\) be a semimetric space. Then the following statements are equivalent.
\begin{enumerate}
\item \label{t4.2:s1} The inequality \(\deg_{G}(v) \leqslant 1\) holds for every vertex \(v\) of every proximinal graph \(G = G_{X, d}(A, B)\).
\item \label{t4.2:s2} The equality
\begin{equation}\label{t4.2:e1}
|A_0| = |B_0|
\end{equation}
holds for all disjoint proximinal sets \(A\), \(B \subseteq X\), where \(A_0\) and \(B_0\) are defined by \eqref{t4.1:e5}.
\item \label{t4.2:s3} For every proximinal \(A \subseteq X\) and every \(x \in X\) there exists the unique best approximation to \(x\) in \(A\).
\item \label{t4.2:s4} For every \(Y \subseteq X\) and every \(x \in X\) there exists at most one best approximation to \(x\) in \(Y\).
\item \label{t4.2:s5} \((X, d) \in \mathbf{WR}\).
\end{enumerate}
\end{theorem}

\begin{proof}
\(\ref{t4.2:s1} \Rightarrow \ref{t4.2:s2}\). Let \ref{t4.2:s1} hold and let \(A\) and \(B\) be disjoint proximinal subsets of \(X\). We must prove \eqref{t4.2:e1}.

Let us consider first the case when \(A_0\) is empty and show that \(B_0 = \varnothing\). Indeed, if \(B_0 \neq \varnothing\) and \(b_0 \in B_0\), then we have
\begin{equation}\label{t4.2:e2}
\dist(b_0, A) = \dist(A, B)
\end{equation}
by~\eqref{t4.1:e5}. Since \(A\) is proximinal, we can find \(a_0 \in A\) such that \(d(b_0, a_0) = \dist(b_0, A)\). The last equality and \eqref{t4.2:e2} imply
\begin{equation}\label{t4.2:e3}
d(a_0, b_0) = \dist(A, B).
\end{equation}
Moreover, we evidently have
\begin{equation}\label{t4.2:e4}
d(a_0, b_0) \geqslant \dist(a_0, B) \geqslant \dist(A, B).
\end{equation}
Now from \eqref{t4.2:e3} and \eqref{t4.2:e4} follows \(\dist(a_0, B) = \dist(A, B)\). Hence, \(a_0 \in A_0\), contrary to \(A_0 = \varnothing\). Similarly, from \(B_0 = \varnothing\) follows \(A_0 = \varnothing\). Thus, it suffices to prove \eqref{t4.2:e1} when \(A_0 \neq \varnothing \neq B_0\).

Let us consider an arbitrary \(v_0 \in B\). Using Definition~\ref{d1.5} we see that the inequality \(\deg_{G}(v_0) > 0\) holds if and only if \(v_0 \in B_0\). Thus, we have \(B_0 = \left\{b \in B \colon \deg_{G}(b) > 0\right\}\) and, analogously, \(A_0 = \left\{a \in A \colon \deg_{G}(a) > 0\right\}\). Consequently, there is a mapping \(F \colon B_0 \to A_0\) such that
\begin{equation}\label{t4.2:e5}
\{v, F(v)\} \in E(G)
\end{equation}
for every \(v \in B_0\). Statement~\ref{t4.2:s1} implies that \(F\) is injective. Indeed, if \(u_0 \in A_0\) and \(F(v_1) = u_0 = F(v_2)\) holds for some different \(v_1\), \(v_2 \in B_0\), then we have \(\deg_{G}(u_0) \geqslant 2\), contrary to \ref{t4.2:s1}.

Since \(F\) is injective, the inequality \(|B_0| \leqslant |A_0|\) holds. Similarly, we obtain \(|A_0| \leqslant |B_0|\). Equality~\eqref{t4.2:e1} follows.

\(\ref{t4.2:s2} \Rightarrow \ref{t4.2:s3}\). Let \ref{t4.2:s2} hold, let \(A\) be a proximinal subset of \(X\) and let \(x\) be an arbitrary point of \(X\). If \(x \in A\), then \(x\) is only best approximation in \(A\) to itself. If \(x\) does not belong to \(A\), then, for \(B := \{x\}\), we have \(B_0 = \{x\}\) by \eqref{t4.1:e5} and Definition~\ref{d1.1}. The sets \(A\) and \(B\) are disjoint proximinal subsets of \(X\) and, consequently,
\begin{equation}\label{t4.2:e6}
|A_0| = |B_0| = 1
\end{equation}
holds by statement~\ref{t4.2:s2}. Since the set \(A_0\) is the set of all best approximations to \(x\) in \(A\), statement \ref{t4.2:s3} follows from \eqref{t4.2:e6} and \eqref{t4.1:e5}.

\(\ref{t4.2:s3} \Rightarrow \ref{t4.2:s4}\). Let \ref{t4.2:s3} hold. If \(Y\) is a subset of \(X\) and \(x\) is a point of \(X\) such that there are distinct \(y_1\), \(y_2 \in Y\) satisfying
\[
d(x, y_1) = d(x, y_2) = \dist(x, Y),
\]
then the set \(A = \{y_1, y_2\}\) is proximinal in \((X, d)\) and \(y_1\), \(y_2\) are two different best approximations to \(x\) in \(A\), contrary to \ref{t4.2:s3}.

\(\ref{t4.2:s4} \Rightarrow \ref{t4.2:s5}\). Let \ref{t4.2:s4} hold. If \((X, d)\) is not weakly rigid, then, by Definition~\ref{d3.1}, there is a three-point \(X' \subseteq X\) such that \(d|_{X' \times X'} \in \mathbf{SR}\). Hence, we can find distinct \(x_1\), \(x_2\), \(x_3 \in X\) for which the equality
\begin{equation}\label{t4.2:e7}
d(x_1, x_2) = d(x_3, x_2)
\end{equation}
holds (see the proof of Proposition~\ref{p2.4}). Let \(Y = \{x_2, x_3\}\). Thus, \(x_1\) and \(x_3\) are two different best approximations to \(x_2\) in \(Y\), contrary to \ref{t4.2:s4}.

\(\ref{t4.2:s5} \Rightarrow \ref{t4.2:s1}\). It follows from Theorem~\ref{t4.1}.
\end{proof}

The membership \((X, d) \in \mathbf{WR}\) can be easily described by digraphs introduced in Definition~\ref{d1.9}. Let us denote by \(Di^0\) and \(Di^1\) the digraphs depicted below on Figure~\ref{fig3}. It is easy to see that \((X, d)\) is weakly rigid iff \(Di_Y\) is isomorphic to \(Di^0\) for every three-point \(Y \subseteq X\). Moreover, \((X, d)\) is strongly rigid iff it is weakly rigid and \(Di_Z\) is isomorphic to \(Di^1\) for every four-point \(Z \subseteq X\).

\begin{figure}[htb]
\centering
\begin{tikzpicture}[scale=1,
arrow/.style = {-{Stealth[length=5pt]}, shorten >=2pt}]
\def\xx{1cm}
\def\yy{-1cm}
\def\dr{2pt}
\coordinate (v1) at (0, 0);
\coordinate (v2) at (0, \yy);
\coordinate (v3) at (0, 2*\yy);
\coordinate (v4) at (0, 3*\yy);
\draw [fill, black] (v1) circle (\dr);
\draw [fill, black] (v2) circle (\dr);
\draw [fill, black] (v3) circle (\dr);
\draw [fill, black] (v4) circle (\dr);
\draw [arrow] (v1) -- (v2);
\draw [arrow] (v2) -- (v3);
\draw [arrow] (v3) -- (v4);
\draw (0, 0.5) node[above] {\(Di^0\)};

\begin{scope}[xshift=4cm]
\coordinate (v1) at (0, 0);
\coordinate (v2) at (0, \yy);
\coordinate (v3) at (0, 2*\yy);
\coordinate (v4) at (0, 3*\yy);
\coordinate (v5) at (0, 4*\yy);
\draw [fill, black] (v1) circle (\dr);
\draw [fill, black] (v2) circle (\dr);
\draw [fill, black] (v3) circle (\dr);
\draw [fill, black] (v4) circle (\dr);
\draw [fill, black] (v5) circle (\dr);
\draw [arrow] (v1) -- (v2);
\draw [arrow] (v2) -- (v3);
\draw [arrow] (v3) -- (v4);
\draw [arrow] (v4) -- (v5);
\draw (0, 0.5) node[above] {\(Di^1\)};
\end{scope}
\end{tikzpicture}
\caption{Digraphs corresponding to all three-point and four-point strongly rigid semimetric spaces.} \label{fig3}
\end{figure}
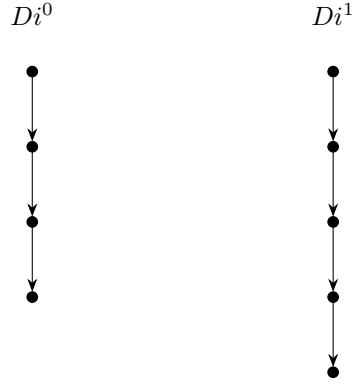

\begin{corollary}\label{c3.3}
The double inclusion
\begin{equation}\label{c3.3:e1}
\mathbf{SR} \subseteq \mathbf{UBPP} \subseteq \mathbf{WR}
\end{equation}
holds.
\end{corollary}

\begin{proof}
The first inclusion in \eqref{c3.3:e1} coincides with inclusion~\eqref{t3.2:e10}. To prove the second one, it suffices to use Theorem~\ref{t4.2} and Theorem~\ref{t3.2} for the case when \(B\) is a one-point set.
\end{proof}

As Example~\ref{ex2.5} shows, the first inclusion in \eqref{c3.3:e1} is strict. A four-point semimetric space
\begin{equation}\label{e3.6}
(X, d) \in \mathbf{WR} \setminus \mathbf{UBPP}
\end{equation}
is constructed in the next example (see Figure~\ref{fig3.1}).

\begin{figure}[htb]
\centering
\begin{tikzpicture}[scale=1,
arrow/.style = {-{Stealth[length=5pt]}, shorten >=2pt}]
\def\xx{1.5cm}
\def\yy{-1.5cm}
\coordinate [label=above left:{$a_1^*$}] (a1) at (0, 0);
\coordinate [label=below left:{$b_2^*$}] (b2) at (0, \yy);
\coordinate [label=above right:{$a_2^*$}] (a2) at (3*\xx, 0);
\coordinate [label=below right:{$b_1^*$}] (b1) at (3*\xx, 1.2*\yy);
\draw (a1) -- node [left] {\(1\)} (b2) -- node [below] {\(4\)} (b1) -- node [right] {\(2\)} (a2) -- node [above] {\(5\)} (a1) -- node [near start, above] {\(3\)} (b1);
\draw (b2) -- node [near start, above] {\(3\)} (a2);
\draw [fill, black] (b1) circle (2pt);
\draw [fill, black] (b2) circle (2pt);
\draw [fill, black] (a1) circle (2pt);
\draw [fill, black] (a2) circle (2pt);
\draw (2, 0.5) node[above] {\((X^*, \rho^*)\)};

\begin{scope}[xshift=9cm]
\def\xx{1cm}
\def\yy{-1cm}
\def\dr{2pt}
\coordinate (v1) at (0, 0);
\coordinate (v2) at (0, \yy);
\coordinate (v3) at (-\xx, 2*\yy);
\coordinate (v4) at (\xx, 2*\yy);
\coordinate (v5) at (0, 3*\yy);
\coordinate (v6) at (0, 4*\yy);
\draw [fill, black] (v1) circle (\dr);
\draw [fill, black] (v2) circle (\dr);
\draw [fill, black] (v3) circle (\dr);
\draw [fill, black] (v4) circle (\dr);
\draw [fill, black] (v5) circle (\dr);
\draw [fill, black] (v6) circle (\dr);
\draw [arrow] (v1) -- (v2);
\draw [arrow] (v2) -- (v3);
\draw [arrow] (v2) -- (v4);
\draw [arrow] (v3) -- (v5);
\draw [arrow] (v4) -- (v5);
\draw [arrow] (v5) -- (v6);
\draw (0, 0.5) node[above] {\(Di_{X^*}\)};
\end{scope}

\begin{scope}[xshift=13cm]
\def\xx{1cm}
\def\yy{-2cm}
\def\dr{2pt}
\coordinate [label=above left:\(a_1^*\)] (a1) at (0, 0);
\coordinate [label=above right:\(a_2^*\)] (a2) at (\xx, 0);
\coordinate [label=below left:\(b_2^*\)] (b2) at (0, \yy);
\coordinate [label=below right:\(b_1^*\)] (b1) at (\xx, \yy);
\draw [fill, black] (a1) circle (\dr);
\draw [fill, black] (a2) circle (\dr);
\draw [fill, black] (b1) circle (\dr);
\draw [fill, black] (b2) circle (\dr);
\draw (a1) -- (b2);
\draw (a2) -- (b1);
\draw (0, 0.5) node[above] {\(G_{X^*}(A^*, B^*)\)};
\end{scope}
\end{tikzpicture}
\caption{The space \((X^*, \rho^*)\), its digraph \(Di_{X^*}\), and the proximinal graph \(G_{X^*}(A^*, B^*)\) for \(A^* = \{a_1^*, b_2^*\}\) and \(B^* = \{a_2^*, b_1^*\}\).} \label{fig3.1}
\end{figure}
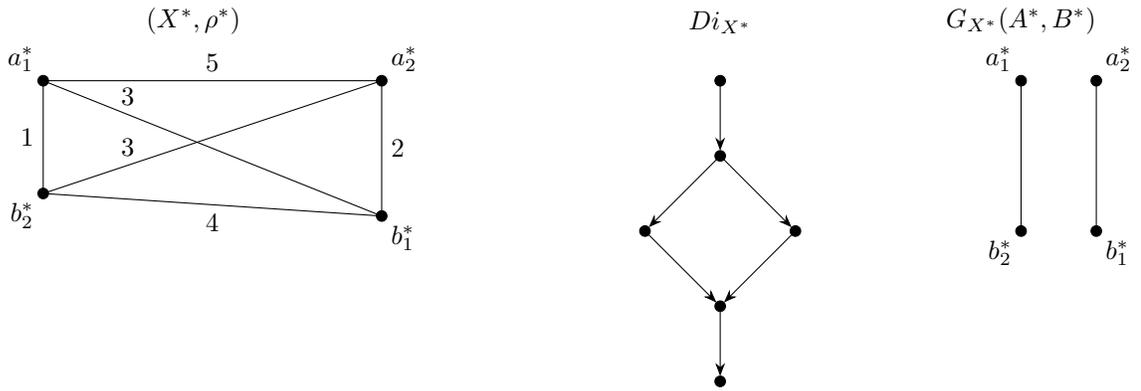

\begin{example}\label{ex3.4}
Let \(X^* = \{a_1^*, b_1^*, a_2^*, b_2^*\}\) be a four-point set and let \(\rho^* \colon X^* \times X^* \to [0, \infty)\) be a semimetric on \(X^*\) such that
\begin{align*}
\rho^*(a_1^*, b_2^*) & = 1, \quad \rho^*(a_2^*, b_1^*) = 2, \quad \rho^*(a_1^*, b_1^*) = \rho^*(a_2^*, b_2^*) = 3,\\
\rho^*(b_1^*, b_2^*) & = 4 \quad \text{and} \quad \rho^*(a_1^*, a_2^*) = 5.
\end{align*}
Then it can be proved directly that \eqref{e3.6} holds with \(X = X^{*}\) and \(d = \rho^*\). Moreover, it will be shown later in Lemma~\ref{l4.5} that, up to weak similarity, the class \(\mathbf{WR} \setminus \mathbf{UBPP}\) contains the only finite semimetric space \(Y\) whose digraph \(Di_Y\) is isomorphic to \(Di_{X^*}\).
\end{example}

In the rest of the section we discuss an interrelation between proximinal subspaces of weakly rigid semimetric spaces and proximinal subspaces of ultrametric spaces.

Recall that a metric space \((Y, \rho)\) is \emph{ultrametric} if the \emph{strong triangle inequality}
\[
d(x, y) \leqslant \max \bigl\{d(x, z), d(z, y)\bigr\}
\]
holds for all \(x\), \(y\), \(z \in Y\). For every ultrametric space \(X\), each triangle in \(X\) is isosceles with the base being no greater than the legs. The converse statement is also valid: If \(Z\) is a semimetric space and each triangle in \(Z\) is isosceles with the base no greater than the legs, then \(Z\) is an ultrametric space. The situation is opposite for weakly rigid spaces. A semimetric space is weakly rigid if and only if it does not contain any isosceles triangles. However, there is a close relationship between proximinal subsets in weakly rigid spaces and in ultrametric ones. To describe and prove it, we need the following.

\begin{theorem}[\cite{CDL2021a}]\label{t3.5}
Let \(G\) be a bipartite graph with fixed parts \(A\) and \(B\), and let
\[
G' = \bigl\{v \in V(G) \colon \deg_{G}(v) > 0\bigr\}.
\]
Then the following statements are equivalent:
\begin{enumerate}
\item \label{t3.5:s1} Either \(G\) is not a null graph and \(G'\) is a disjoint union of complete bipartite graphs, or \(E(G) = \varnothing\), but the sets \(A\) and \(B\) are infinite.
\item \label{t3.5:s2} \(G\) is proximinal for an ultrametric space \((X, d)\) with \(X = A \cup B\).
\end{enumerate}
\end{theorem}

\begin{proposition}\label{p3.6}
Let \((X, d) \in \mathbf{WR}\), let \(A\) be a proximinal subset of \(X\) such that \(X \setminus A \neq \varnothing\), and let \(\Gamma = \Gamma_{X, d}(A)\) be a bipartite graph with the parts \(A\) and \(B = X \setminus A\). If, for all \(a \in A\) and \(b \in B\), we have
\begin{equation}\label{p3.6:e1}
\bigl(\{a, b\} \in E(\Gamma_{X, d}(A))\bigr) \Leftrightarrow \bigl(d(a, b) = \dist(b, A)\bigr),
\end{equation}
then there is an ultrametric \(\rho \colon X \times X \to [0, \infty)\) such that
\begin{equation}\label{p3.6:e2}
\Gamma_{X, d}(A) = G_{X, \rho}(A, B).
\end{equation}
\end{proposition}

\begin{proof}
Let \eqref{p3.6:e1} hold for all \(a \in A\) and \(b \in B\). By Theorem~\ref{t3.5}, it suffices to show that \(\Gamma\) is not a null graph,
\begin{equation}\label{p3.6:e3}
E(\Gamma) \neq \varnothing,
\end{equation}
and \(\Gamma' := \bigl\{v \in V(\Gamma) \colon \deg_{\Gamma}(v) > 0\bigr\}\) is a disjoint union of complete bipartite graphs.

First of all we note that Theorem~\ref{t4.1} together with \eqref{p3.6:e1} implies \eqref{p3.6:e3} because \(A\) is proximinal in \((X, d)\). Thus, it suffices to show that \(\Gamma'\) is a disjoint union of complete bipartite graphs.

Let us consider an arbitrary \(\{a^*, b^*\} \in \Gamma'\), \(a^* \in A\), \(b^* \in B\). Since \(\Gamma'\) is a subgraph of \(\Gamma\) and \((X, d) \in \mathbf{WR}\), Theorem~\ref{t4.1} and \eqref{p3.6:e1} imply that \(a^*\) is the unique best approximation to \(b^*\) in \(A\). Write \(B^* = B^*(a^*)\) for the set of all points \(b \in B\) which satisfies the relationship \(\{a^*, b\} \in E(\Gamma')\). Then the graph \(\Gamma'_{a^*}\) with
\begin{equation}\label{p3.6:e4}
V(\Gamma'_{a^*}) = \{a^*\} \cup B^*(a^*) \quad \text{and} \quad E(\Gamma'_{a^*}) = \bigl\{\{a^*, b\} \colon b \in B^*(a^*)\bigr\}
\end{equation}
is a subgraph of \(\Gamma'\), and
\[
V(\Gamma'_{a_1^*}) \cap V(\Gamma'_{a_2^*}) = \varnothing
\]
holds whenever \(a_1^*\) and \(a_2^*\) are different points of \(A \cap V(\Gamma')\). Thus, \(\Gamma'\) is the disjoint union of the graphs \(\Gamma'_{a}\), \(a \in A \cap V(\Gamma')\). Using \eqref{p3.6:e4} we see that every \(\Gamma'_{a}\) is a star with the center \(a\) whenever \(a \in A \cap V(\Gamma')\). This completes the proof because each star is a complete bipartite graph.
\end{proof}

\begin{remark}\label{r4.7}
A special kind of bipartite graphs, the trees, gives a natural language for description of ultrametric spaces \cite{Carlsson2010, DLW, Fie, GV2012DAM, HolAMM2001, H04, BH2, Lemin2003, Bestvina2002, DDP2011pNUAA, DP2019PNUAA, DPT2017FPTA, DPT2015, Pet2018pNUAA, DP2018pNUAA, Dov2020TaAoG, BS2017, DP2020pNUAA, DKa2021, Dov2019pNUAA, DP2013SM, PD2014JMS}, but the authors are aware of only papers \cite{BDK2021a} and \cite{PD2014JMS}, in which complete bipartite and, more generally, complete multipartite graphs are systematically used to study ultrametric spaces.
\end{remark}

We conclude this section by following.

\begin{conjecture}\label{con4.8}
Let \(\Gamma\) be a bipartite graph with fixed parts \(A\) and \(B\). Then the following statements are equivalent.
\begin{enumerate}
\item \label{con4.8:s1} There is a weakly rigid semimetric \(d\) on \(X = A \cup B\) such that \(A\) is proximinal in \((X, d)\), and \(\Gamma = \Gamma_{X, d}(A)\), and
\[
\bigl(\{a, b\} \in E(\Gamma_{X, d}(A))\bigr) \Leftrightarrow \bigl(d(a, b) = \dist(b, A)\bigr)
\]
holds for all \(a \in A\) and \(b \in B\).
\item \label{con4.8:s2} Every connected component of \(\Gamma\) is a star with a center \(a \in A\), and the equality \(\deg_{\Gamma} b = 1\) holds for every \(b \in B\), and the inequality
\[
|V(\Gamma)| \leqslant \mathfrak{c}
\]
holds, where \(\mathfrak{c}\) is the cardinality of the continuum.
\end{enumerate}
\end{conjecture}

\section{Uniqueness of the best proximity pair}
\label{sec5}

The goal of the section is to describe the semimetric spaces belonging to \(\mathbf{UBPP}\).

\begin{lemma}\label{l2.6}
Let \((Y, d) \in \mathbf{UBPP}\) be a four-point semimetric space. Then the inequality
\begin{equation}\label{l2.6:e1}
|D(Y)| \geqslant 5
\end{equation}
holds; and there is a unique vertex \(v^1 = \{x^1, y^1\} \in V(Di_Y)\) such that
\begin{equation}\label{l2.6:e2}
d(x^1, y^1) = \min D(Y);
\end{equation}
and there is a unique \(v^2 = \{x^2, y^2\} \in V(Di_Y)\) such that \((v^2, v^1)\) is an arc of \(Di_Y\),
\begin{equation}\label{l2.6:e3}
(v^2, v^1) \in E(Di_Y).
\end{equation}
\end{lemma}

\begin{proof}
Since \(D(Y)\) is finite and nonempty, it follows from Definition~\ref{d1.9} that there is a vertex \(\{x^1, y^1\}\) of \(Di_Y\) for which \eqref{l2.6:e2} holds. Suppose we can find \(\{x_*^1, y_*^1\} \in V(Di_Y)\) such that \(d(x_*^1, y_*^1) = d(x^1, y^1)\) and \(\{x_*^1, y_*^1\} \neq \{x^1, y^1\}\). If the sets \(\{x_*^1, y_*^1\}\) and \(\{x^1, y^1\}\) are disjoint, then these sets are different edges of the proximinal graph \(G_Y(A, B)\) with \(A = \{x_*^1, x^1\}\) and \(B = \{y_*^1, y^1\}\) contrary to \((Y, d) \in \mathbf{UBPP}\). For the case when \(\{x_*^1, y_*^1\} \cap \{x^1, y^1\} \neq \varnothing\), we see that \(Y\) contains an ``isosceles triangle'', contrary to Proposition~\ref{p2.4}. Thus, there is the unique \(v^1 = \{x^1, y^1\}\) which satisfies \eqref{l2.6:e2}.

Let us prove the existence and uniqueness of \(v^2 = \{x^2, y^2\} \in V(Di_Y)\) which satisfies \eqref{l2.6:e3}. Since \(v^1\) is unique, we can find \(x\), \(y \in Y\) such that \(d(x, y) > d(x^1, y^1)\). Thus, \(D(Y) \setminus \{d(x^1, y^1)\}\) is finite and nonempty. Hence, \(D(Y) \setminus \{d(x^1, y^1)\}\) contains the upper cover \(t^*\) of the number \(d(x^1, y^1)\). Let \(x^2\), \(y^2\) be some points of \(Y\) for which \(t^* = d(x^2, y^2)\) holds. Then, by Definition~\ref{d1.9}, we have \eqref{l2.6:e3} for \(v^2 = \{x^2, y^2\}\).

Suppose now that there is \(\{x_*^2, y_*^2\}\) such that
\begin{equation}\label{l2.6:e5}
\{x_*^2, y_*^2\} \neq \{x^2, y^2\}.
\end{equation}
and \eqref{l2.6:e3} holds with \(v^2 = \{x_*^2, y_*^2\}\). As before, using Proposition~\ref{p2.4} and \eqref{l2.6:e5} we can show that \(\{x_*^2, y_*^2\} \cap \{x^2, y^2\} = \varnothing\). Hence, the equality \(Y = \{x_*^2, y_*^2\} \cup \{x^2, y^2\}\) holds. We evidently have \(\{x^2, y^2\} \neq \{x^1, y^1\} \neq \{x_*^2, y_*^2\}\). Consequently, there are points \(a_1\), \(a_1^* \in Y\) such that
\[
\{a_1\} = \{x^2, y^2\} \setminus \{x^1, y^1\} \quad \text{and} \quad \{a_1^*\} = \{x_*^2, y_*^2\} \setminus \{x^1, y^1\}.
\]
Write \(A = \{a_1, a_1^*\}\) and \(B = \{x^1, y^1\}\). Then the proximinal graph \(G_Y(A, B)\) contains the different edges \(\{x^2, y^2\}\) and \(\{x_*^2, y_*^2\}\) contrary to \((Y, d) \in \mathbf{UBPP}\). Thus, there is the unique \(v^2 \in V(Di_Y)\) which satisfies \eqref{l2.6:e3}.

Let us prove inequality \eqref{l2.6:e1}. If \eqref{l2.6:e1} does not hold, then we have
\begin{equation}\label{l2.6:e6}
|D(Y)| \leqslant 4.
\end{equation}
We will show that \eqref{l2.6:e6} implies \((Y, d) \notin \mathbf{UBPP}\).

Suppose first that
\begin{equation}\label{l2.6:e7}
\{x^1, y^1\} \cap \{x^2, y^2\} \neq \varnothing,
\end{equation}
where \(\{x^1, y^1\} = v^1\) and \(\{x^2, y^2\} = v^2\) satisfy \eqref{l2.6:e2} and \eqref{l2.6:e3}, respectively. Since \(v^1\) and \(v^2\) are different vertices of \(Di_Y\), we may also assume that \(y^1 \neq y^2\), which together with \eqref{l2.6:e7} implies the equality \(x^1 = x^2\). The set \(Y \setminus \{x^1, y^1, y^2\}\) contains the unique point \(p\). Let us consider the family \(F = \{d(p, x^1), d(p, y^1), d(p, y^2)\}\) of all distances from the point \(p\) to the points \(x^1\), \(y^1\), \(y^2\) (see Figure~\ref{fig4}).

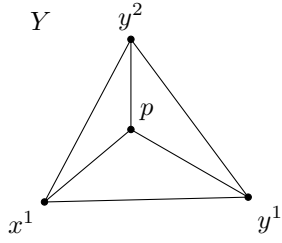
\begin{figure}[h]
\centering
\begin{tikzpicture}[scale=0.6]
\coordinate [label=below left:{$x^1$}] (A) at (220:2.5cm);
\coordinate [label=below right:{$y^1$}] (B) at (-30:3cm);
\coordinate [label=above:{$y^2$}] (C) at (90:2cm);
\coordinate [label=above right:{$p$}] (P) at (0, 0);
\draw (A) -- (B) -- (C) -- (A);
\draw (A) -- (P) -- (B);
\draw (C) -- (P);
\draw [fill, black] (A) circle (2pt);
\draw [fill, black] (B) circle (2pt);
\draw [fill, black] (C) circle (2pt);
\draw [fill, black] (P) circle (2pt);

\draw (-2, 2) node[above] {\(Y\)};
\end{tikzpicture}
\caption{The edges \(\{x^1, y^1\}\) and \(\{x^2, y^2\}\) of the complete graph \(K_{|Y|}\) are no parallel.}\label{fig4}
\end{figure}

The uniqueness \(v^1 = \{x^1, y^1\}\) and \(v^2 = \{x^2, y^2\}\) implies that all these distances are strictly greater than \(d(x^1, y^1)\) and \(d(x^2, y^2)\). Consequently, using inequality~\eqref{l2.6:e6}, we see that \(F\) contains at most two different elements. Hence, at least one from the ``triangles'' \((p, x^1, y^1)\), \((p, x^1, y^2)\) and \((p, y^2, y^1)\) is isosceles, which contradicts \((Y, d) \in \mathbf{UBPP}\) by Proposition~\ref{p2.4}. Thus, \(\{x^1, y^1\}\) and \(\{x^2, y^2\}\) are disjoint, \(\{x^1, y^1\} \cap \{x^2, y^2\} = \varnothing\). In this case, instead of the ``tetrahedron'' depicted by Figure~\ref{fig4}, we consider the ``quadrangle'' with diagonals \(\{x^1, y^1\}\) and \(\{x^2, y^2\}\) (see Figure~\ref{fig5}).

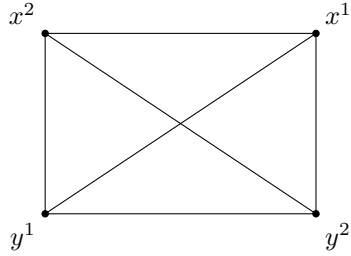
\begin{figure}[h]
\centering
\begin{tikzpicture}[scale=0.6]
\def\xx{3}
\def\yy{2}
\coordinate [label=above right:{$x^1$}] (x1) at (\xx, \yy);
\coordinate [label=below left:{$y^1$}] (y1) at (-\xx, -\yy);
\coordinate [label=above left:{$x^2$}] (x2) at (-\xx, \yy);
\coordinate [label=below right:{$y^2$}] (y2) at (\xx, -\yy);
\draw (x1) -- (x2) -- (y1) -- (y2) -- (x1) -- (y1);
\draw (x2) -- (y2);
\draw [fill, black] (x1) circle (2pt);
\draw [fill, black] (x2) circle (2pt);
\draw [fill, black] (y1) circle (2pt);
\draw [fill, black] (y2) circle (2pt);
\end{tikzpicture}
\caption{The edges \(\{x^1, y^1\}\) and \(\{x^2, y^2\}\) of the complete graph \(K_{|Y|}\) are parallel.}\label{fig5}
\end{figure}

As above, we make show that the family \(\{d(x^1, x^2), d(x^2, y^1), d(y^1, y^2), d(y^2, x^1)\}\) contains at most two different elements. Moreover, Proposition~\ref{p2.4} implies
\[
d(x^1, x^2) \neq d(x^2, y^1) \neq d(y^1, y^2) \neq d(y^2, x^1) \neq d(x^1, x^2).
\]
Hence, the equalities
\begin{equation}\label{l2.6:e8}
d(x^1, x^2) = d(y^1, y^2) \quad \text{and} \quad d(x^2, y^1) = d(y^2, x^1)
\end{equation}
hold. Let us consider now the proximinal graph \(G_Y(A, B)\) with \(A = \{x^1, y^1\}\) and \(B = \{x^2, y^2\}\). Then, using \eqref{l2.6:e8}, it is easy to show that \(G_Y(A, B)\) has at least two edges, contrary to \((Y, d) \in \mathbf{UBPP}\).
\end{proof}

In what follows, we will denote by \(Di^1\), \(Di^2\), \(Di^3\), \(Di^4\) the corresponding digraphs depicted by Figure~\ref{fig6}.

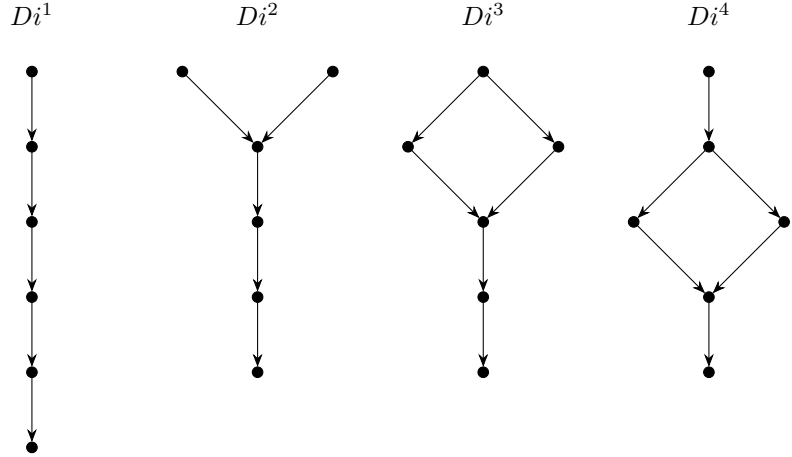
\begin{figure}[htb]
\centering
\begin{tikzpicture}[scale=1,
arrow/.style = {-{Stealth[length=5pt]}, shorten >=2pt}]
\def\xx{1cm}
\def\yy{-1cm}
\def\dr{2pt}
\coordinate (v1) at (0, 0);
\coordinate (v2) at (0, \yy);
\coordinate (v3) at (0, 2*\yy);
\coordinate (v4) at (0, 3*\yy);
\coordinate (v5) at (0, 4*\yy);
\coordinate (v6) at (0, 5*\yy);
\draw [fill, black] (v1) circle (\dr);
\draw [fill, black] (v2) circle (\dr);
\draw [fill, black] (v3) circle (\dr);
\draw [fill, black] (v4) circle (\dr);
\draw [fill, black] (v5) circle (\dr);
\draw [fill, black] (v6) circle (\dr);
\draw [arrow] (v1) -- (v2);
\draw [arrow] (v2) -- (v3);
\draw [arrow] (v3) -- (v4);
\draw [arrow] (v4) -- (v5);
\draw [arrow] (v5) -- (v6);
\draw (0, 0.5) node[above] {\(Di^1\)};

\begin{scope}[xshift=3cm]
\coordinate (v1) at (-\xx, 0);
\coordinate (v2) at (\xx, 0);
\coordinate (v3) at (0, \yy);
\coordinate (v4) at (0, 2*\yy);
\coordinate (v5) at (0, 3*\yy);
\coordinate (v6) at (0, 4*\yy);
\draw [fill, black] (v1) circle (\dr);
\draw [fill, black] (v2) circle (\dr);
\draw [fill, black] (v3) circle (\dr);
\draw [fill, black] (v4) circle (\dr);
\draw [fill, black] (v5) circle (\dr);
\draw [fill, black] (v6) circle (\dr);
\draw [arrow] (v1) -- (v3);
\draw [arrow] (v2) -- (v3);
\draw [arrow] (v3) -- (v4);
\draw [arrow] (v4) -- (v5);
\draw [arrow] (v5) -- (v6);
\draw (0, 0.5) node[above] {\(Di^2\)};
\end{scope}

\begin{scope}[xshift=6cm]
\coordinate (v1) at (0, 0);
\coordinate (v2) at (-\xx, \yy);
\coordinate (v3) at (\xx, \yy);
\coordinate (v4) at (0, 2*\yy);
\coordinate (v5) at (0, 3*\yy);
\coordinate (v6) at (0, 4*\yy);
\draw [fill, black] (v1) circle (\dr);
\draw [fill, black] (v2) circle (\dr);
\draw [fill, black] (v3) circle (\dr);
\draw [fill, black] (v4) circle (\dr);
\draw [fill, black] (v5) circle (\dr);
\draw [fill, black] (v6) circle (\dr);
\draw [arrow] (v1) -- (v2);
\draw [arrow] (v1) -- (v3);
\draw [arrow] (v2) -- (v4);
\draw [arrow] (v3) -- (v4);
\draw [arrow] (v4) -- (v5);
\draw [arrow] (v5) -- (v6);
\draw (0, 0.5) node[above] {\(Di^3\)};
\end{scope}

\begin{scope}[xshift=9cm]
\coordinate (v1) at (0, 0);
\coordinate (v2) at (0, \yy);
\coordinate (v3) at (-\xx, 2*\yy);
\coordinate (v4) at (\xx, 2*\yy);
\coordinate (v5) at (0, 3*\yy);
\coordinate (v6) at (0, 4*\yy);
\draw [fill, black] (v1) circle (\dr);
\draw [fill, black] (v2) circle (\dr);
\draw [fill, black] (v3) circle (\dr);
\draw [fill, black] (v4) circle (\dr);
\draw [fill, black] (v5) circle (\dr);
\draw [fill, black] (v6) circle (\dr);
\draw [arrow] (v1) -- (v2);
\draw [arrow] (v2) -- (v3);
\draw [arrow] (v2) -- (v4);
\draw [arrow] (v3) -- (v5);
\draw [arrow] (v4) -- (v5);
\draw [arrow] (v5) -- (v6);
\draw (0, 0.5) node[above] {\(Di^4\)};
\end{scope}
\end{tikzpicture}
\caption{The digraphs \(Di_Y\) of four-point spaces \((Y, d) \in \mathbf{UBPP}\).} \label{fig6}
\end{figure}

\begin{lemma}\label{t2.7}
Let \((Y, d) \in \mathbf{UBPP}\) be a four-point semimetric space. Then the digraph \(Di_{Y}\) is isomorphic to one of the digraphs \(Di^1\), \(Di^2\), \(Di^3\), \(Di^4\).
\end{lemma}

\begin{proof}
If \(Y\) is strongly rigid, then it follows directly from Definitions~\ref{d1.0} and \ref{d1.5} that \(Di_Y\) is isomorphic to the digraph \(Di^1\).

Let us consider the case when \(Y\) is not strongly rigid. We must prove that \(Di_Y\) is isomorphic to one of the digraphs \(Di^2\), \(Di^3\), \(Di^4\). Let \(D(Y)\) be the set of all nonzero distances between points of \(Y\). Since \(|Y| = 4\) holds, we have \(|D(Y)| \leqslant 6\). Moreover, Lemma~\ref{l2.6} implies the inequality \(|D(Y)| \geqslant 5\). Consequently, the equality
\begin{equation}\label{t2.7:e1}
|D(Y)| = 5
\end{equation}
holds, i.e., we have \(D(Y) = \{d^1, \ldots, d^5\}\) for some \(d^i \in (0, \infty)\), \(i = 1, \ldots, 5\). Without loss of generality, we can assume \(d^1 < d^2 < \ldots < d^5\). Using~Lemma~\ref{l2.6} we see that there is the unique vertex \(v^1 = \{x^1, y^1\} \in V(Di_Y)\) such that \(d(x^1, y^1) = d^1\) and, in addition, there is the unique \(v^2 = \{x^2, y^2\} \in V(Di_Y)\) such that \(d(x^2, y^2) = d^2\). Since \(Y\) is a four-point set, we have
\begin{equation}\label{t2.7:e2}
|V(Di_Y)| = |E(K_{|Y|})| = |E(K_{|4|})| = 6.
\end{equation}

We claim that, for every \(i = 3\), \(4\), \(5\), the number of all different \(\{x, y\} \in E(K_{|Y|})\), which satisfy the equality \(d(x, y) = d^i\), does not exceed two and this number is equal to two only for one of \(d^3\), \(d^4\), \(d^5\).

Indeed, suppose contrary that there are \(k \in \{3, 4, 5\}\) and pairwise different \(\{x_1, y_1\}\), \(\{x_2, y_2\}\) and \(\{x_3, y_3\}\) such that \(d(x_1, y_1) = d(x_2, y_2) = d(x_3, y_3) = d^k\). Then, using the pigeonhole principle, we can find \(z \in Y\) and \(x\), \(y \in Y\) such that \((z, y, x)\) is a ``isosceles triangle'' with two equal sides having length \(d^k\). It is a contradiction with \((Y, d) \in \mathbf{UBPP}\).

If we can find pairwise distinct \(\{\ol{x}_1, \ol{y}_1\}\), \(\{\ol{x}_2, \ol{y}_2\}\), \(\{\ol{x}_3, \ol{y}_3\}\), \(\{\ol{x}_4, \ol{y}_4\} \in K_{|Y|}\) and different \(i_1\), \(i_2 \in \{3, 4, 5\}\) such that
\[
d(\ol{x}_1, \ol{y}_1) = d(\ol{x}_2, \ol{y}_2) = d^{i_1} \quad \text{and} \quad d(\ol{x}_3, \ol{y}_3) = d(\ol{x}_4, \ol{y}_4) = d^{i_2},
\]
then, using equality~\eqref{t2.7:e2}, we obtain the inequality \(|D(Y)| \leqslant 4\), which contradicts~\eqref{t2.7:e1}. Hence, the above formulated claim is valid. It implies, in particular, that the digraph \(Di_Y\) is isomorphic to the one of the digraphs \(Di^2\), \(Di^3\), \(Di^4\).
\end{proof}

\begin{lemma}\label{l3.7}
Let \((Y, d) \in \mathbf{WR}\) be a four-point semimetric space. If \(Di_Y\) is isomorphic to the one of the digraphs \(Di^1\), \(Di^2\), \(Di^3\), then \((Y, d)\) belongs to \(\mathbf{UBPP}\).
\end{lemma}

\begin{proof}
If \(Di_Y\) is isomorphic to \(Di^1\), then \(|D(Y)| = 6\) holds and, consequently, we have
\begin{equation}\label{l3.7:e1}
|D(Y)| = |V(Di_Y)|.
\end{equation}
Using Definition~\ref{d1.0} and the formula \(D(Y) = \bigl\{d(x, y) \colon x \neq y \text{ and } x, y \in Y\bigr\}\), we see that \eqref{l3.7:e1} holds if and only if \((Y, d)\) is strongly rigid. Hence, \((Y, d)\) belongs to \(\mathbf{UBPP}\) by Corollary~\ref{c3.3}.

Let \(Di_Y\) be isomorphic to \(Di^2\). We must show that each proximinal graph \(G_Y(A, B)\) has exactly one edge. If \(A \cup B\) is a proper subset of \(Y\), then it follows from Proposition~\ref{p2.4} because \((Y, d) \in \mathbf{WR}\). Now if \(A \cup B = Y\), then the set \(D_{A, B} := \bigl\{d(a, b) \colon a \in A \text{ and } b \in B\bigr\}\) contains at least three and at most four elements,
\begin{equation}\label{l3.7:e2}
3 \leqslant |D_{A, B}| \leqslant 4.
\end{equation}
It is clear that \(\dist (A, B)\) is the smallest element in \(D_{A, B}\). Double inequality~\eqref{l3.7:e2} and \(A \cap B = \varnothing\) imply
\begin{equation}\label{l3.7:e3}
\dist (A, B) < \diam (Y),
\end{equation}
where \(\diam (Y) = \max\{d(x, y) \colon x, y \in Y\}\). Since \(Di_Y\) and \(Di^2\) are isomorphic, there is a unique element \(\ol{d} \in D(Y)\) such that \(\ol{d}\) is equal to \(\diam (Y)\) and \(d(x, y) = \ol{d} = d(u, v)\) for distinct \(\{x, y\}\), \(\{u, v\} \in V(Di_Y)\). Consequently, \eqref{l3.7:e3} implies that there are the unique \(a_0 \in A\) and the unique \(b_0 \in B\) such that \(d(a_0, b_0) = \dist (A, B)\). Thus, \(\{a_0, b_0\}\) is the unique edge of \(G_X(A, B)\).

For the case when \(Di_Y\) is isomorphic to \(Di^3\), the membership \((Y, d) \in \mathbf{UBPP}\) can be shown in the same way as in the previous case.
\end{proof}

To describe the situation when a four-point weakly rigid \((Y, d)\) belongs to \(\mathbf{UBPP}\) and \(Di_Y\) is isomorphic to \(Di^4\), we will use the concept of weak similarity introduced in Definition~\ref{d1.11}.

\begin{lemma}\label{l4.4}
Let \((X, d)\) and \((Y, \rho)\) be weakly similar semimetric spaces with a weak similarity \(\Phi \colon X \to Y\). Then, for every proximinal graph \(G_{X, d}(A, B)\), the sets \(\Phi(A)\) and \(\Phi(B)\) are disjoint proximinal subsets of \(Y\), and the proximinal graphs \(G_{X, d}(A, B)\) and \(G_{Y, \rho}(\Phi(A), \Phi(B))\) are isomorphic, and the restriction
\[
\Phi|_{A \cup B} \colon A \cup B \to \Phi(A) \cup \Phi(B)
\]
is an isomorphism of the graphs \(G_{X, d}(A, B)\) and \(G_{Y, \rho}(\Phi(A), \Phi(B))\). In particular, the equivalences
\[
\bigl((X, d) \in \mathbf{UBPP}\bigr) \Leftrightarrow \bigl((Y, \rho) \in \mathbf{UBPP}\bigr) \quad \text{and}\quad \bigl((X, d) \in \mathbf{WR}\bigr) \Leftrightarrow \bigl((Y, \rho) \in \mathbf{WR}\bigr)
\]
are valid.
\end{lemma}

The proof is straightforward. We only note that the inverse mapping \(\Phi^{-1}\colon Y \to X\) also is a weak similarity.

In the next lemma and in Theorem~\ref{t4.5} below, we denote by \((X^*, \rho^*)\) the four-point semimetric space from Example~\ref{ex3.4}.

\begin{lemma}\label{l4.5}
Let \((Y, d) \in \mathbf{WR}\) be a four-point semimetric space and let \(Di_Y\) be isomorphic to the digraph \(Di^4\). Then \((Y, d)\) belongs to \(\mathbf{UBPP}\) if and only if \((Y, d)\) is not weakly similar to the semimetric space \((X^*, \rho^*)\).
\end{lemma}

\begin{proof}
Let \((Y, d)\) and \((X^*, \rho^*)\) be weakly similar. By Lemma~\ref{l4.4}, we obtain
\begin{equation}\label{l4.5:e1}
(Y, d) \notin  \mathbf{UBPP}
\end{equation}
because there is a proximinal graph \(G_{X^*, \rho^*}(A^*, B^*)\) with two different edges (see Figure~\ref{fig3.1}).

Conversely, let \eqref{l4.5:e1} hold. We must prove that \((Y, d)\) and \((X^*, \rho^*)\) are weakly similar. Since \(Di_Y\) and \(Di^4\) are isomorphic digraphs, the set \(D(Y)\) has exactly five elements,
\[
D(Y) = \{d^1, \ldots, d^5\},
\]
and we may assume that \(d^1 < d^2 < \ldots < d^5\). Using isomorphism of \(Di_Y\) and \(Di^4\) again, we see that there are exactly two different \(\{a_1, b_1\}\), \(\{a_2, b_2\} \in E(K_{|Y|})\) such that \(d(a_1, b_1) = d(a_2, b_2) = d^3\) and, in addition, if \(j \in \{1, \ldots, 5\}\) and \(j \neq 3\), then there is exactly one \(\{x_j, y_j\} \in E(K_{|Y|})\) which satisfies the equality \(d(x_j, y_j) = d^j\). Consequently, if \(\{x, y\}\) and \(\{u, v\}\) are distinct edges of \(K_{|Y|}\) and \(d(x, y) = d(u, v)\), then the equality
\[
\bigl\{\{a_1, b_1\}, \{a_2, b_2\}\bigr\} = \bigl\{\{x, y\}, \{u, v\}\bigr\}
\]
holds. Using~\eqref{l4.5:e1} we can find nonempty disjoint sets \(A\), \(B \subseteq Y\) such that \(G_Y(A, B)\) has at least two different edges \(\{y_1, y_2\}\) and \(\{y_3, y_4\}\). Since \(G_Y(A, B)\) is proximinal, the equality
\[
d(y_1, y_2) = d(y_3, y_4)
\]
holds. As was shown above, the last equality implies
\[
\bigl\{\{y_1, y_2\}, \{y_3, y_4\}\bigr\} = \bigl\{\{a_1, b_1\}, \{a_2, b_2\}\bigr\}.
\]
From \((Y, d) \in \mathbf{WR}\) and \(d(a_1, b_1) = d(a_2, b_2)\) follows
\begin{equation}\label{l4.5:e2}
\{a_1, b_1\} \cap \{a_2, b_2\} = \varnothing.
\end{equation}
Consequently, we have
\begin{equation}\label{l4.5:e3}
E(G_Y(A, B)) = \bigl\{\{a_1, b_1\}, \{a_2, b_2\}\bigr\}.
\end{equation}
Equalities \eqref{l4.5:e2} and \eqref{l4.5:e3} imply
\begin{equation}\label{l4.5:e4}
|A| \geqslant 2 \quad \text{and} \quad |B| \geqslant 2.
\end{equation}
Since \(A \cap B = \varnothing\) and \(A \cup B \subseteq Y\), from \eqref{l4.5:e4} follows that \(|A| = |B| = 2\) and, consequently,
\begin{equation}\label{l4.5:e5}
V(G_Y(A, B)) = A \cup B = Y.
\end{equation}

\begin{figure}[htb]
\centering
\begin{tikzpicture}[scale=1,
arrow/.style = {-{Stealth[length=5pt]}, shorten >=2pt}]
\def\xx{1.2cm}
\def\yy{0.7cm}
\def\dr{2pt}

\coordinate [label=above left:{$a_1$}] (a1) at (-\xx, \yy);
\coordinate [label=above right:{$a_2$}] (a2) at (\xx, \yy);
\coordinate [label=below right:{$b_1$}] (b1) at (\xx, -\yy);
\coordinate [label=below left:{$b_2$}] (b2) at (-\xx, -\yy);
\draw [fill, black] (a1) circle (\dr);
\draw [fill, black] (a2) circle (\dr);
\draw [fill, black] (b1) circle (\dr);
\draw [fill, black] (b2) circle (\dr);
\draw (a1) -- node [near start, above] {\(d^3\)} (b1);
\draw (a2) -- node [near start, above] {\(d^3\)} (b2);
\draw (a1) -- node [left] {\(A\)} (b2);
\draw (a2) -- node [right] {\(B\)} (b1);

\begin{scope}[xshift=5cm]
\coordinate [label=above left:{$a_1$}] (a1) at (-\xx, \yy);
\coordinate [label=above right:{$a_2$}] (a2) at (\xx, \yy);
\coordinate [label=below right:{$b_1$}] (b1) at (\xx, -\yy);
\coordinate [label=below left:{$b_2$}] (b2) at (-\xx, -\yy);
\draw [fill, black] (a1) circle (\dr);
\draw [fill, black] (a2) circle (\dr);
\draw [fill, black] (b1) circle (\dr);
\draw [fill, black] (b2) circle (\dr);
\draw (a1) -- node [near start, below] {\(d^3\)} (b1);
\draw (a2) -- node [near start, below] {\(d^3\)} (b2);
\draw (a1) -- node [above] {\(A\)} (a2);
\draw (b1) -- node [below] {\(B\)} (b2);
\end{scope}
\end{tikzpicture}
\caption{Step~1. Choice of parts \(A\) and \(B\).} \label{fig7}
\end{figure}
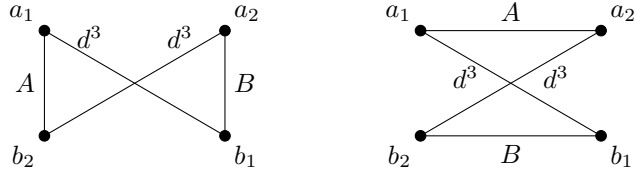

By renaming \(A\) and \(B\), if necessary, we can assume \(a_1 \in A\) and \(b_1 \in B\). Then, using \eqref{l4.5:e2}, \eqref{l4.5:e3}, \eqref{l4.5:e5} and \(A \cap B = \varnothing\), we can prove that the following only two cases are possible:
\begin{align}
\label{l4.5:e6}
A & = \{a_1, a_2\} \quad \text{and} \quad B = \{b_1, b_2\}\\
\intertext{or}
\label{l4.5:e7}
A & = \{a_1, b_2\} \quad \text{and} \quad B = \{b_1, a_2\}
\end{align}
(see Figure~\ref{fig7}). Moreover, it is easy to prove that we have either
\begin{align}
\label{l4.5:e8}
\diam A = d^1 \quad \text{and} \quad \diam B = d^2
\intertext{or}
\label{l4.5:e9}
\diam A = d^2 \quad \text{and} \quad \diam B = d^1.
\end{align}
Indeed, if neither \eqref{l4.5:e8} nor \eqref{l4.5:e9} are fulfilled, then the graph \(G_Y(A, B)\) has exactly one edge, contrary to \eqref{l4.5:e3}. Thus, we only have the following four possible cases (see Figure~\ref{fig8}).

\begin{figure}[htb]
\centering
\begin{tikzpicture}[scale=1,
arrow/.style = {-{Stealth[length=5pt]}, shorten >=2pt}]
\def\xx{1.2cm}
\def\yy{0.7cm}
\def\dr{2pt}

\coordinate [label=above left:{$a_1$}] (a1) at (-\xx, \yy);
\coordinate [label=above right:{$a_2$}] (a2) at (\xx, 1.2*\yy);
\coordinate [label=below right:{$b_1$}] (b1) at (\xx, -1.2*\yy);
\coordinate [label=below left:{$b_2$}] (b2) at (-\xx, -\yy);
\draw [fill, black] (a1) circle (\dr);
\draw [fill, black] (a2) circle (\dr);
\draw [fill, black] (b1) circle (\dr);
\draw [fill, black] (b2) circle (\dr);
\draw (a1) -- node [near start, above] {\(d^3\)} (b1);
\draw (b2) -- node [near start, below] {\(d^3\)} (a2);
\draw (a1) -- node [left] {\(d^1\)} (b2);
\draw (a2) -- node [right] {\(d^2\)} (b1);
\draw (-2*\xx, \yy) node[above left] {\(G^1\)};

\begin{scope}[xshift=7cm]
\coordinate [label=above left:{$a_1$}] (a1) at (-\xx, 1.2*\yy);
\coordinate [label=above right:{$a_2$}] (a2) at (\xx, \yy);
\coordinate [label=below right:{$b_1$}] (b1) at (\xx, -\yy);
\coordinate [label=below left:{$b_2$}] (b2) at (-\xx, -1.2*\yy);
\draw [fill, black] (a1) circle (\dr);
\draw [fill, black] (a2) circle (\dr);
\draw [fill, black] (b1) circle (\dr);
\draw [fill, black] (b2) circle (\dr);
\draw (a1) -- node [near start, above] {\(d^3\)} (b1);
\draw (b2) -- node [near start, below] {\(d^3\)} (a2);
\draw (a1) -- node [left] {\(d^2\)} (b2);
\draw (a2) -- node [right] {\(d^1\)} (b1);
\draw (-2*\xx, \yy) node[above left] {\(G^2\)};
\end{scope}

\begin{scope}[yshift=-3.5cm]
\coordinate [label=above left:{$a_1$}] (a1) at (-\xx, \yy);
\coordinate [label=above right:{$a_2$}] (a2) at (\xx, \yy);
\coordinate [label=below right:{$b_1$}] (b1) at (1.2*\xx, -\yy);
\coordinate [label=below left:{$b_2$}] (b2) at (-1.2*\xx, -\yy);
\draw [fill, black] (a1) circle (\dr);
\draw [fill, black] (a2) circle (\dr);
\draw [fill, black] (b1) circle (\dr);
\draw [fill, black] (b2) circle (\dr);
\draw (a1) -- node [near end, above] {\(d^3\)} (b1);
\draw (a2) -- node [near end, above] {\(d^3\)} (b2);
\draw (a1) -- node [above] {\(d^1\)} (a2);
\draw (b1) -- node [below] {\(d^2\)} (b2);
\draw (-2*\xx, \yy) node[above left] {\(G^3\)};
\end{scope}

\begin{scope}[xshift=7cm, yshift=-3.5cm]
\coordinate [label=above left:{$a_1$}] (a1) at (-1.2*\xx, \yy);
\coordinate [label=above right:{$a_2$}] (a2) at (1.2*\xx, \yy);
\coordinate [label=below right:{$b_1$}] (b1) at (\xx, -\yy);
\coordinate [label=below left:{$b_2$}] (b2) at (-\xx, -\yy);
\draw [fill, black] (a1) circle (\dr);
\draw [fill, black] (a2) circle (\dr);
\draw [fill, black] (b1) circle (\dr);
\draw [fill, black] (b2) circle (\dr);
\draw (a1) -- node [near start, below] {\(d^3\)} (b1);
\draw (a2) -- node [near start, below] {\(d^3\)} (b2);
\draw (a1) -- node [above] {\(d^2\)} (a2);
\draw (b1) -- node [below] {\(d^1\)} (b2);
\draw (-2*\xx, \yy) node[above left] {\(G^4\)};
\end{scope}
\end{tikzpicture}
\caption{Step~2. Choice of weights on \(A\) and \(B\).} \label{fig8}
\end{figure}
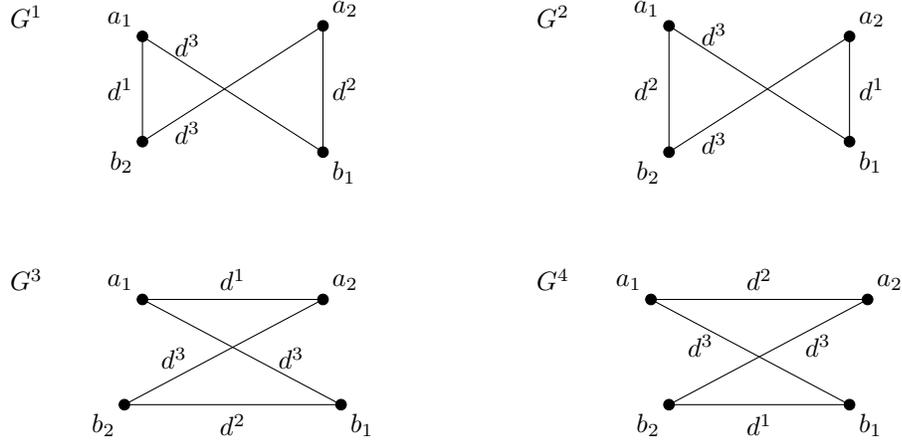

We can now completely describe all admissible extensions of the weights \(E(G^l) \to D(Y)\), \(l = 1, \ldots, 4\), to the weights \(E(K_{|Y|}) \to D(Y) \) and, consequently, to the semimetrics \(Y \times Y \to D_Y \to [0, \infty)\). To do this, note that in each of the graphs \(G^1\), \(\ldots\), \(G^4\) there are exactly two pairs of nonadjacent points of \(Y\) with different pairwise distances equal \(d^4\) or \(d^5\). For example, starting from \(G^1\), we obtain the admissible semimetrics \(d^{1,1}\) and \(d^{1,2}\). Similarly, using \(G^2\), we obtain the semimetrics \(d^{2,1}\) and \(d^{2,2}\), and so on (see Figure~\ref{fig10}).

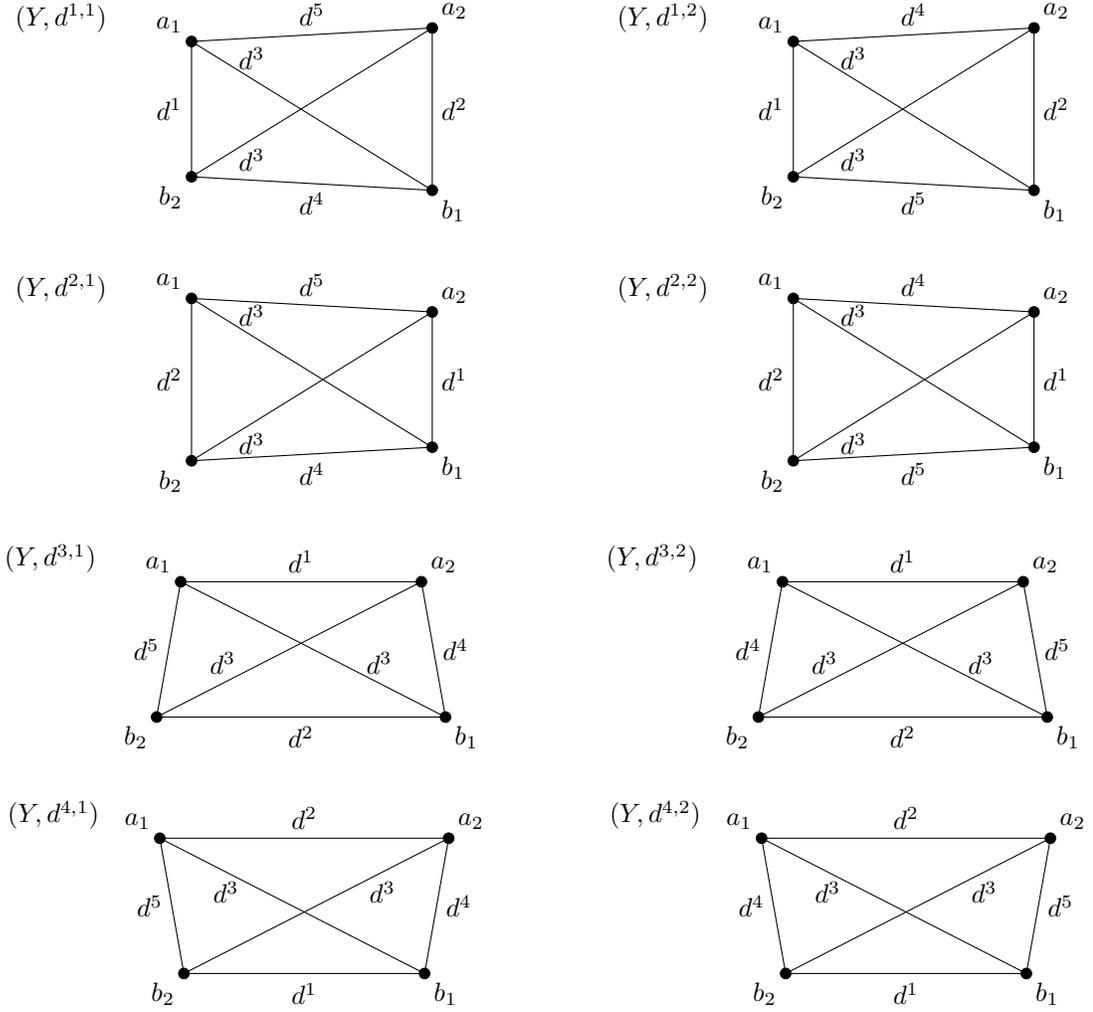
\begin{figure}[ht]
\centering
\begin{tikzpicture}[scale=1,
arrow/.style = {-{Stealth[length=5pt]}, shorten >=2pt}]
\def\xx{1.6cm}
\def\yy{0.9cm}
\def\dr{2pt}
\def\xshift{8cm}
\def\yshift{-4cm}

\coordinate [label=above left:{$a_1$}] (a1) at (-\xx, \yy);
\coordinate [label=above right:{$a_2$}] (a2) at (\xx, 1.2*\yy);
\coordinate [label=below right:{$b_1$}] (b1) at (\xx, -1.2*\yy);
\coordinate [label=below left:{$b_2$}] (b2) at (-\xx, -\yy);
\draw [fill, black] (a1) circle (\dr);
\draw [fill, black] (a2) circle (\dr);
\draw [fill, black] (b1) circle (\dr);
\draw [fill, black] (b2) circle (\dr);
\draw (a1) -- node [near start, above] {\(d^3\)} (b1);
\draw (b2) -- node [near start, below] {\(d^3\)} (a2);
\draw (a1) -- node [left] {\(d^1\)} (b2);
\draw (a2) -- node [right] {\(d^2\)} (b1);
\draw (a1) -- node [above] {\(d^5\)} (a2);
\draw (b1) -- node [below] {\(d^4\)} (b2);
\draw (-\xx-1cm, \yy) node[above left] {\((Y, d^{1,1})\)};

\begin{scope}[xshift=\xshift]
\coordinate [label=above left:{$a_1$}] (a1) at (-\xx, \yy);
\coordinate [label=above right:{$a_2$}] (a2) at (\xx, 1.2*\yy);
\coordinate [label=below right:{$b_1$}] (b1) at (\xx, -1.2*\yy);
\coordinate [label=below left:{$b_2$}] (b2) at (-\xx, -\yy);
\draw [fill, black] (a1) circle (\dr);
\draw [fill, black] (a2) circle (\dr);
\draw [fill, black] (b1) circle (\dr);
\draw [fill, black] (b2) circle (\dr);
\draw (a1) -- node [near start, above] {\(d^3\)} (b1);
\draw (b2) -- node [near start, below] {\(d^3\)} (a2);
\draw (a1) -- node [left] {\(d^1\)} (b2);
\draw (a2) -- node [right] {\(d^2\)} (b1);
\draw (a1) -- node [above] {\(d^4\)} (a2);
\draw (b1) -- node [below] {\(d^5\)} (b2);
\draw (-\xx-1cm, \yy) node[above left] {\((Y, d^{1,2})\)};
\end{scope}
\end{tikzpicture}
\bigskip

\begin{tikzpicture}[scale=1,
arrow/.style = {-{Stealth[length=5pt]}, shorten >=2pt}]
\def\xx{1.6cm}
\def\yy{0.9cm}
\def\dr{2pt}
\def\xshift{8cm}
\def\yshift{-4cm}

\coordinate [label=above left:{$a_1$}] (a1) at (-\xx, 1.2*\yy);
\coordinate [label=above right:{$a_2$}] (a2) at (\xx, \yy);
\coordinate [label=below right:{$b_1$}] (b1) at (\xx, -\yy);
\coordinate [label=below left:{$b_2$}] (b2) at (-\xx, -1.2*\yy);
\draw [fill, black] (a1) circle (\dr);
\draw [fill, black] (a2) circle (\dr);
\draw [fill, black] (b1) circle (\dr);
\draw [fill, black] (b2) circle (\dr);
\draw (a1) -- node [near start, above] {\(d^3\)} (b1);
\draw (b2) -- node [near start, below] {\(d^3\)} (a2);
\draw (a1) -- node [left] {\(d^2\)} (b2);
\draw (a2) -- node [right] {\(d^1\)} (b1);
\draw (a1) -- node [above] {\(d^5\)} (a2);
\draw (b1) -- node [below] {\(d^4\)} (b2);
\draw (-\xx-1cm, \yy) node[above left] {\((Y, d^{2,1})\)};

\begin{scope}[xshift=\xshift]
\coordinate [label=above left:{$a_1$}] (a1) at (-\xx, 1.2*\yy);
\coordinate [label=above right:{$a_2$}] (a2) at (\xx, \yy);
\coordinate [label=below right:{$b_1$}] (b1) at (\xx, -\yy);
\coordinate [label=below left:{$b_2$}] (b2) at (-\xx, -1.2*\yy);
\draw [fill, black] (a1) circle (\dr);
\draw [fill, black] (a2) circle (\dr);
\draw [fill, black] (b1) circle (\dr);
\draw [fill, black] (b2) circle (\dr);
\draw (a1) -- node [near start, above] {\(d^3\)} (b1);
\draw (b2) -- node [near start, below] {\(d^3\)} (a2);
\draw (a1) -- node [left] {\(d^2\)} (b2);
\draw (a2) -- node [right] {\(d^1\)} (b1);
\draw (a1) -- node [above] {\(d^4\)} (a2);
\draw (b1) -- node [below] {\(d^5\)} (b2);
\draw (-\xx-1cm, \yy) node[above left] {\((Y, d^{2,2})\)};
\end{scope}
\end{tikzpicture}
\bigskip

\begin{tikzpicture}[scale=1,
arrow/.style = {-{Stealth[length=5pt]}, shorten >=2pt}]
\def\xx{1.6cm}
\def\yy{0.9cm}
\def\dr{2pt}
\def\xshift{8cm}
\def\yshift{-4cm}

\coordinate [label=above left:{$a_1$}] (a1) at (-\xx, \yy);
\coordinate [label=above right:{$a_2$}] (a2) at (\xx, \yy);
\coordinate [label=below right:{$b_1$}] (b1) at (1.2*\xx, -\yy);
\coordinate [label=below left:{$b_2$}] (b2) at (-1.2*\xx, -\yy);
\draw [fill, black] (a1) circle (\dr);
\draw [fill, black] (a2) circle (\dr);
\draw [fill, black] (b1) circle (\dr);
\draw [fill, black] (b2) circle (\dr);
\draw (a1) -- node [near end, above] {\(d^3\)} (b1);
\draw (a2) -- node [near end, above] {\(d^3\)} (b2);
\draw (a1) -- node [above] {\(d^1\)} (a2);
\draw (b1) -- node [below] {\(d^2\)} (b2);
\draw (a1) -- node [left] {\(d^5\)} (b2);
\draw (a2) -- node [right] {\(d^4\)} (b1);
\draw (-\xx-1cm, \yy) node[above left] {\((Y, d^{3,1})\)};

\begin{scope}[xshift=\xshift]
\coordinate [label=above left:{$a_1$}] (a1) at (-\xx, \yy);
\coordinate [label=above right:{$a_2$}] (a2) at (\xx, \yy);
\coordinate [label=below right:{$b_1$}] (b1) at (1.2*\xx, -\yy);
\coordinate [label=below left:{$b_2$}] (b2) at (-1.2*\xx, -\yy);
\draw [fill, black] (a1) circle (\dr);
\draw [fill, black] (a2) circle (\dr);
\draw [fill, black] (b1) circle (\dr);
\draw [fill, black] (b2) circle (\dr);
\draw (a1) -- node [near end, above] {\(d^3\)} (b1);
\draw (a2) -- node [near end, above] {\(d^3\)} (b2);
\draw (a1) -- node [above] {\(d^1\)} (a2);
\draw (b1) -- node [below] {\(d^2\)} (b2);
\draw (a1) -- node [left] {\(d^4\)} (b2);
\draw (a2) -- node [right] {\(d^5\)} (b1);
\draw (-\xx-1cm, \yy) node[above left] {\((Y, d^{3,2})\)};
\end{scope}
\end{tikzpicture}
\bigskip

\begin{tikzpicture}[scale=1,
arrow/.style = {-{Stealth[length=5pt]}, shorten >=2pt}]
\def\xx{1.6cm}
\def\yy{0.9cm}
\def\dr{2pt}
\def\xshift{8cm}
\def\yshift{-4cm}

\coordinate [label=above left:{$a_1$}] (a1) at (-1.2*\xx, \yy);
\coordinate [label=above right:{$a_2$}] (a2) at (1.2*\xx, \yy);
\coordinate [label=below right:{$b_1$}] (b1) at (\xx, -\yy);
\coordinate [label=below left:{$b_2$}] (b2) at (-\xx, -\yy);
\draw [fill, black] (a1) circle (\dr);
\draw [fill, black] (a2) circle (\dr);
\draw [fill, black] (b1) circle (\dr);
\draw [fill, black] (b2) circle (\dr);
\draw (a1) -- node [near start, below] {\(d^3\)} (b1);
\draw (a2) -- node [near start, below] {\(d^3\)} (b2);
\draw (a1) -- node [above] {\(d^2\)} (a2);
\draw (b1) -- node [below] {\(d^1\)} (b2);
\draw (a1) -- node [left] {\(d^5\)} (b2);
\draw (a2) -- node [right] {\(d^4\)} (b1);
\draw (-\xx-1cm, \yy) node[above left] {\((Y, d^{4,1})\)};

\begin{scope}[xshift=\xshift]
\coordinate [label=above left:{$a_1$}] (a1) at (-1.2*\xx, \yy);
\coordinate [label=above right:{$a_2$}] (a2) at (1.2*\xx, \yy);
\coordinate [label=below right:{$b_1$}] (b1) at (\xx, -\yy);
\coordinate [label=below left:{$b_2$}] (b2) at (-\xx, -\yy);
\draw [fill, black] (a1) circle (\dr);
\draw [fill, black] (a2) circle (\dr);
\draw [fill, black] (b1) circle (\dr);
\draw [fill, black] (b2) circle (\dr);
\draw (a1) -- node [near start, below] {\(d^3\)} (b1);
\draw (a2) -- node [near start, below] {\(d^3\)} (b2);
\draw (a1) -- node [above] {\(d^2\)} (a2);
\draw (b1) -- node [below] {\(d^1\)} (b2);
\draw (a1) -- node [left] {\(d^4\)} (b2);
\draw (a2) -- node [right] {\(d^5\)} (b1);
\draw (-\xx-1cm, \yy) node[above left] {\((Y, d^{4,2})\)};
\end{scope}
\end{tikzpicture}
\caption{Step~3. All possible weights on the edges of the graph \(K_{|Y|}\).} \label{fig10}
\end{figure}

Thus, there are \(i \in \{1, \ldots, 4\}\) and \(j \in \{1, 2\}\) such that \(d = d^{i, j}\).

We claim that all semimetric spaces \((Y, d^{i, j})\), \(i \in \{1, \ldots, 4\}\), \(j \in \{1, 2\}\), are pairwise isometric. To construct the desirable isometries \(\Phi \colon (Y, d^{i_1, j_1}) \to (Y, d^{i_2, j_2})\), we note that, for every \(i \in \{1, \ldots, 4\}\) and every \(j \in \{1, 2\}\), there is exactly one \(a_1^{i, j} \in Y\) such that
\begin{align}
\label{l4.5:e10}
d^{i, j}(a_1^{i, j}, x) &= d^1,\\
\intertext{and}
\label{l4.5:e11}
d^{i, j}(a_1^{i, j}, y) &= d^5
\end{align}
hold for some \(x\), \(y \in Y\). In particular, we have
\[
a_1^{1,1} = a_1^{3,1} = a_1, \quad a_1^{1,2} = a_1^{4,1} = b_2, \quad a_1^{2,1} = a_1^{3,2} = a_2, \quad a_1^{2,2} = a_1^{4,2} = b_1.
\]
Moreover, equations \eqref{l4.5:e10} and \eqref{l4.5:e11} have the unique solutions which we denote by \(a_2^{i, j}\) and \(a_3^{i, j}\), respectively. Since \(d^1 \neq d^5\), we have \(a_2^{i, j} \neq a_3^{i, j}\). Let \(a_4^{i, j}\) be the unique point of the set
\[
Y \setminus \{a_1^{i, j}, a_2^{i, j}, a_3^{i, j}\}.
\]
Then simple direct calculations show that, for given \(i_1\), \(i_2 \in \{1, \ldots, 4\}\) and \(j_1\), \(j_2 \in \{1, 2\}\), the mapping \(\Phi \colon Y \to Y\),
\[
\Phi(a_k^{i_1, j_1}) = a_k^{i_2, j_2}, \quad k \in \{1, \ldots, 4\},
\]
is an isometry of the semimetric spaces \((Y, d^{i_1, j_1})\) and \((Y, d^{i_2, j_2})\).

To complete the proof it suffices to note that the mapping \(F \colon X^* \to Y\) satisfying the equalities
\[
F(a_1^*) = a_1, \quad F(a_2^*) = a_2, \quad F(b_1^*) = b_1 \quad \text{and} \quad F(b_2^*) = b_2
\]
is a weak similarity of \((X^*, \rho^*)\) and \((Y, d^{1,1})\).
\end{proof}

The following theorem is the main result of the section.

\begin{theorem}\label{t4.5}
Let \((X, d)\) be a semimetric space. Then the following statements are equivalent:
\begin{enumerate}
\item \label{t4.5:s1} \((X, d) \in \mathbf{UBPP}\).
\item \label{t4.5:s2} \((X, d) \in \mathbf{WR}\), and, for every four-point \(Y \subseteq X\), the digraph \(Di_Y\) is isomorphic to the one of the digraphs \(Di^1\), \(Di^2\), \(Di^3\), \(Di^4\), and \((X, d)\) does not contain any four-point subspace which is weakly similar to the semimetric space \((X^*, \rho^*)\).
\end{enumerate}
\end{theorem}

\begin{proof}
\(\ref{t4.5:s1} \Rightarrow \ref{t4.5:s2}\). Let \((X, d)\) belong to \(\mathbf{UBPP}\). Then we have \((X, d) \in \mathbf{WR}\) by Corollary~\ref{c3.3}, and, for every four-point \(Y \subseteq X\), the digraph \(Di_Y\) is isomorphic to the one of the digraphs \(Di^1\), \(\ldots\), \(Di^4\) by Lemma~\ref{t2.7}.

Suppose that there is a four-point \(Y \subseteq X\) such that \((Y, d|_{Y \times Y})\) is weakly similar to \((X^*, \rho^*)\). Since \((X^*, \rho^*)\) does not belong to \(\mathbf{UBPP}\) (see Example~\ref{ex3.4}), we obtain
\begin{equation}\label{t4.5:e1}
(Y, d|_{Y \times Y}) \notin \mathbf{UBPP}
\end{equation}
by Lemma~\ref{l4.4}. It follows directly from the definition of the class \(\mathbf{UBPP}\) that every subspace of any \(\mathbf{UBPP}\) space also belongs to \(\mathbf{UBPP}\). Hence, \eqref{t4.5:e1} contradicts to \((X, d) \in \mathbf{UBPP}\). Statement \ref{t4.5:s2} follows.

\(\ref{t4.5:s2} \Rightarrow \ref{t4.5:s1}\). Let \ref{t4.5:s2} hold. We must show that
\begin{equation}\label{t4.5:e2}
(X, d) \in \mathbf{UBPP}.
\end{equation}
The last relationship follows from Definition~\ref{d3.1} and Proposition~\ref{p2.4} if \(|X| \leqslant 3\) holds.

Let us consider the case when \(|X| > 3 \). In this case \eqref{t4.5:e2} holds if and only if, for every four-point \(Y \subseteq X\), we have
\begin{equation}\label{t4.5:e3}
(Y, d|_{Y \times Y}) \in \mathbf{UBPP}.
\end{equation}
Indeed, suppose we have \eqref{t4.5:e3} for every four-point \(Y \subseteq X\), but \eqref{t4.5:e2} does not hold. Then there are disjoint proximinal subsets \(A\) and \(B\) of \(X\) such that the proximinal graph \(G_X(A, B)\) contains at least two distinct edges \(\{a_1, b_1\}\) and \(\{a_2, b_2\}\), \(a_i \in A\), \(b_i \in B\), \(i = 1\), \(2\). Write
\begin{equation}\label{t4.5:e4}
A^0 = \{a_1\} \cup \{a_2\} \quad \text{and} \quad B^0 = \{b_1\} \cup \{b_2\}.
\end{equation}
Since every nonempty finite subset of \((X, d)\) is proximinal in \((X, d)\), we see that \(A^0\) and \(B^0\) are disjoint proximinal subsets of \((X, d)\). Moreover, using \eqref{e1.2}, \eqref{t4.5:e4} and the inclusions \(A^0 \subseteq A\), \(B^0 \subseteq B\), we obtain
\[
\dist(A^0, B^0) \geqslant \dist (A, B) = d(a_1, b_1) = d(a_2, b_2) \geqslant \dist(A^0, B^0).
\]
Hence, \(\{a_1, b_1\}\) and \(\{a_2, b_2\}\) are also the edges of the proximinal graph \(G_{Y, \rho}(A^0, B^0)\) for \(Y = A^0 \cup B^0\) and \(\rho = d|_{Y \times Y}\). From \((X, d) \in \mathbf{WR}\) it follows that \((Y, \rho) \in \mathbf{WR}\). The last two statements and \((Y, \rho) \in \mathbf{WR}\) imply that \(Y\) contains exactly four points and \((Y, d|_{Y \times Y})\) does not belong to \(\mathbf{UBPP}\), contrary to \eqref{t4.5:e3}.

To complete the proof it suffices to note that \eqref{t4.5:e3} holds for every four-point \(Y \subseteq X\) by statement~\ref{t4.5:s2} and Lemmas~\ref{l3.7} and \ref{l4.5}.
\end{proof}

\begin{corollary}
Let \((X, d)\) be a semimetric space. Then \((X, d) \in \mathbf{UBPP}\) if and only if we have \((Y, d|_{Y \times Y}) \in \mathbf{UBPP}\) for every \(Y \subseteq X\) with \(|Y| \leqslant 4\).
\end{corollary}

The following problems lead to the future development of the main results of the paper, Theorems~\ref{t4.2} and \ref{t4.5}.

\begin{problem}
Describe the structure of semimetric spaces \((X, d)\) for which every \(x \in X\) has at most \(k\) best approximations in every proximinal \(A \subseteq X\) with a given integer \(k \geqslant 1\).
\end{problem}

\begin{problem}
Describe the structure of semimetric spaces \((X, d)\) for which every proximinal graph \(G_{X, d}(A, B)\) has at most \(k\) edges with given integer \(k \geqslant 1\).
\end{problem}

\section*{Funding}

Oleksiy Dovgoshey was partially supported by Volkswagen Stiftung Project ``From
Modeling and Analysis to Approximation''.

\medskip

\textbf{Declaration of competing interest}
\medskip

No conflicts of interest to be disclosed.


\end{document}